\documentclass[12pt]{amsart}
\usepackage{amssymb,amsfonts,amsmath}
\newcommand{\K}{\mathbb{C}}

\newcommand{\g}{\mathfrak{g}}
\newcommand{\Spec}{\operatorname{Spec}}
\newcommand{\Z}{\mathbb{Z}}
\newcommand{\A}{\mathcal{A}}
\newcommand{\gr}{\operatorname{gr}}

\newcommand\M{\mathcal{M}}
\newcommand\Loc{\operatorname{Loc}}
\newcommand\Ext{\operatorname{Ext}}

\newcommand\GL{\operatorname{GL}}

\newcommand\I{\mathcal{I}}

\newcommand\J{\mathcal{J}}

\newcommand\End{\operatorname{End}}

\newcommand\Hom{\operatorname{Hom}}

\renewcommand\a{\mathfrak{a}}

\newcommand\quo{/\!/}

\newcommand\C{\mathbb{C}}

\newcommand{\CW}{\mathfrak{CW}}

\newcommand\param{\mathfrak{P}}

\newcommand\HC{\operatorname{HC}}

\newcommand\ZZ{\mathbb{Z}}
\newcommand{\CC}{\mathsf{CC}}

\newcommand{\WC}{\mathfrak{WC}}

\newcommand{\R}{\mathbb{R}}
\newcommand{\Q}{\mathbb{Q}}
\newcommand{\VA}{\operatorname{V}}
\newcommand{\B}{\mathcal{B}}

\newcommand{\Cat}{\mathcal{C}}
\newcommand{\OCat}{\mathcal{O}}

\newcommand{\Ca}{\mathsf{C}}
\newcommand{\Ring}{\mathcal{R}}
\newcommand{\bnabla}{\overline{\nabla}}
\newcommand{\bDelta}{\overline{\Delta}}
\newcommand{\Supp}{\operatorname{Supp}}
\newcommand{\Coh}{\operatorname{Coh}}
\newtheorem{Thm}{Theorem}[section]
\newtheorem{Prop}[Thm]{Proposition}
\newtheorem{Cor}[Thm]{Corollary}
\newtheorem{Lem}[Thm]{Lemma}
\theoremstyle{definition}

\newtheorem{defi}[Thm]{Definition}
\newtheorem{Rem}[Thm]{Remark}

\numberwithin{equation}{section}
\author{Ivan Losev}
\address{Department
of Mathematics, Northeastern University, Boston MA 02115 USA}
\email{i.loseu@neu.edu}
\title{Wall-crossing functors for quantized symplectic resolutions: perversity and partial Ringel dualities}
\begin{document}
\begin{abstract}
In this paper we study wall-crossing functors between categories of modules over
quantizations of symplectic resolutions. We prove that wall-crossing functors
through faces are perverse equivalences and use this to verify an Etingof type conjecture
for quantizations of Nakajima quiver varieties associated to affine quivers.
In the case when there is a Hamiltonian torus action on the resolution
with finitely many fixed points
so that it makes sense to speak about categories $\mathcal{O}$ over quantizations,
we introduce new standardly stratified structures on these categories $\mathcal{O}$
and relate the wall-crossing functors to the Ringel duality functors associated
to these standardly stratified structures.
\end{abstract}
\maketitle
\markright{WALL-CROSSING FUNCTORS FOR QUANTIZED SYMPLECTIC RESOLUTIONS}
\section{Introduction}
In this paper we study the wall-crossing (a.k.a. twisting) functors between categories of modules
over quantizations of symplectic resolutions. These functors are derived equivalences introduced
in this generality in \cite[Section 6.4]{BPW} and further studied in \cite{BL}.

More precisely, let $X^0$ denote a normal Poisson variety admitting a symplectic resolution
of singularities. We also assume that $X^0$ is conical in the sense that it comes with a
$\C^\times$-action that contracts $X^0$ to a single point and rescales the Poisson bracket.
The symplectic resolutions of $X^0$ are parameterized by cones (to be called {\it chambers})
of a certain rational hyperplane arrangement in $H^2(X,\R)$, where $X$ is any of these resolutions. By $X^\theta$
we denote a resolution corresponding to the open cone containing a generic element
$\theta\in H^2(X,\R)$. We can speak about filtered quantizations $\A_\lambda^\theta$
of (the structure sheaf of) $X^\theta$, where $\lambda\in H^2(X,\C)$ is a quantization parameter.
Further, it makes sense to speak about the category of coherent $\A_\lambda^\theta$-modules
to be denoted by $\Coh(\A_\lambda^\theta)$. For $\theta,\theta'$ lying in two different
chambers we have a derived equivalence $\WC_{\theta'\leftarrow \theta}: D^b(\Coh(\A_\lambda^\theta))
\xrightarrow{\sim}D^b(\Coh(\A_\lambda^{\theta'}))$ to be called the {\it wall-crossing functor}.
We will be interested in the special situation when either $\theta,\theta'$
or $\theta,-\theta'$ lie in two chambers
that are opposite to one another with respect to their common face. The two extremes
here is when these chambers  share a common codimension $1$ face and when one is the negative
of the other.

There are two important results about the wall-crossing functors
in this situation that we will obtain in this paper. First, we will show that
the wall-crossing functor for the cones opposite with respect to
a common face is perverse in the sense of Chuang and Rouquier and give some
description of the corresponding filtration. This result was obtained in
some special cases in \cite[Sections 4,7]{BL} and in the closely related setting of
rational Cherednik algebras in \cite{rouq_der}. The proof in the present situation
closely follows \cite{rouq_der} but we need to replace some missing ingredients
such as the restriction functors for Harish-Chandra bimodules.

The perversity of the wall-crossing functors is used to prove a generalization
of the main result of \cite{BL}, Etingof's conjecture on the number of finite
dimensional irreducible modules for symplectic reflection algebras. In \cite{BL}, Bezrukavnikov  and the author re-interpreted the conjecture
in terms of quantized Nakajima quiver varieties. The conjecture was proved
for quivers of finite type and also of affine type with special framing.
In this paper we give a proof for affine quivers with an arbitrary framing\footnote{The proof of this generalization
was previously obtained in \cite{count_affine} but the proof in the present paper is much
simpler and more straightforward.}.

Our second main result concerns the situation when there is a Hamiltonian torus $T$
acting on $X$ with finitely many fixed points. In this case one can fix a generic
one-parameter subgroup $\nu:\C^\times\rightarrow T$ and consider the corresponding
category $\mathcal{O}_\nu(\A_\lambda^\theta)\subset \operatorname{Coh}(\A_\lambda^\theta)$,
introduced in this generality in \cite{BLPW}. This is a highest weight category
whose simple objects are labelled by $X^T$. In \cite{CWR} we have introduced
{\it compatible standardly stratified structures} on $\OCat_{\nu}(\A_\lambda^\theta)$
that roughly speaking come from degenerating $\nu$. In this paper we will produce a
new kind of compatible standardly stratified structures that come from deforming
$\lambda$. The structures are labelled by the faces of the chamber containing $\theta$.
Namely, let $\Gamma$ be a face of the chamber containing $\theta$ and
$\chi$ be an integral point in the interior of $\Gamma$. Set $\theta':=\theta-N\chi$
for $N\gg 0$. We show that the wall-crossing functor $\WC_{\theta\leftarrow \theta'}^{-1}$
is the Ringel duality functor coming from the standardly stratified structure given
by $\theta$.

The paper is organized as follows. Section \ref{S_prelim} recalls various preliminaries
on symplectic resolutions and their quantizations following mostly \cite{BPW}
and \cite{BL}. These preliminaries include the structure of resolutions and their quantizations,
Harish-Chandra bimodules, localization theorems and wall-crossing functors.
In Section \ref{S_perv} we will state and prove a theorem on a perversity of
wall-crossing functors through faces. For this, we will need to study classical
and quantum slices to symplectic leaves and restriction functors for
Harish-Chandra bimodules. We use the perversity to prove an Etingof-type conjecture
for quantized quiver varieties associated to quivers of affine types.
Finally, in Section \ref{S_Ringel} we deal with standardly stratified
categories and their Ringel dualities. The most nontrivial part is to introduce
the standardly stratified structures on $\OCat(\A_\lambda^\theta)$ coming from
deforming $\lambda$. Then we introduce Ringel duality for standardly stratified
categories and show that wall-crossing functors as in the previous paragraph
give Ringel duality functors.

{\bf Acknowledgements}. I am grateful to Roman Bezrukavnikov, Andrei Negut, Andrei Okounkov,
Raphael Rouquier and Ben Webster for stimulating discussions.   This work was
partially supported by the NSF under grants DMS-1161584, DMS-1501558.

\section{Preliminaries}\label{S_prelim}
\subsection{Symplectic resolutions}
Let $X$ be a smooth symplectic variety (with form $\omega$) equipped also with an action of
$\C^\times$ subject to the following conditions:
\begin{itemize}
\item[(a)] There is a positive integer $d$ such that $t.\omega=t^d\omega$.
\item[(b)] The algebra $\C[X]$ is finitely generated and is positively graded: there are
no negative components and the zero component consists of scalars.
\item[(c)] The natural morphism $\rho:X\rightarrow X^0:=\operatorname{Spec}(\C[X])$
is a projective resolution of singularities.
\end{itemize}

We will say that $X$ is a conical symplectic resolution.
Thanks to (b), we can talk about the point $0\in X^0$. Conditions (b) and (c)
imply that $\lim_{t\rightarrow 0}t.x$ exists and lies in $\rho^{-1}(0)$.
So we will call the $\C^\times$-actions on $X$ and $X^0$ {\it contracting}.

By the Grauert-Riemenschneider theorem, we have $H^i(X,\mathcal{O}_X)=0$ for $i>0$.
By results of Kaledin, \cite[Theorem 2.3]{Kaledin}, $X^0$ has finitely many symplectic
leaves.

We will be interested in {\it conical deformations}  $X_{\param}/\param$, where $\param$ is a finite dimensional vector space,
and $X_{\param}$ is a symplectic scheme over $\param$ with symplectic form $\widehat{\omega}\in
\Omega^2(X_\param/\param)$ that specializes to $\omega$ and also with a $\C^\times$-action
on $X_{\param}$ having the following properties:
\begin{itemize}
\item the morphism $X_{\param}\rightarrow \param$ is $\C^\times$-equivariant, where we consider
the action of $\C^\times$ on $\param$ given by $t.p=t^{-d}p$,
\item the restriction of the action to $X$ coincides with the contracting action,
\item $t.\widehat{\omega}:=t^{d}\widehat{\omega}$.
\end{itemize}
It turns out that there is a universal conical deformation $\tilde{X}$ over $\tilde{\param}:=H^2(X,\C)$
(any other deformation is obtained via the pull-back with respect to a linear map
$\param\rightarrow \tilde{\param}$). We will often write $X_{\tilde{\param}}$ instead
of $\tilde{X}$.

For $\lambda\in \tilde{\param}$, let us write $X_{\lambda}$ for the corresponding fiber of $X_{\tilde{\param}}
\rightarrow \tilde{\param}$. 
If $X,X'$ are two conical symplectic resolutions of $X$, then there are
open subvarieties $\breve{X}\subset X, \breve{X}'\subset X'$ with $\operatorname{codim}_X X\setminus \breve{X},
\operatorname{codim}_X X\setminus \breve{X}\geqslant 2$ and $\breve{X}\xrightarrow{\sim} \breve{X}'$,
see, e.g., \cite[Proposition 2.19]{BPW}. This allows to identify the Picard groups $\operatorname{Pic}(X)=\operatorname{Pic}(X')$. Moreover, the Chern class map defines an isomorphism $\C\otimes_{\Z}\operatorname{Pic}(X)\xrightarrow{\sim} H^2(X,\C)$,
See, e.g., \cite[Section 2.3]{BPW}. Let $\tilde{\param}_{\Z}$ be the image of $\operatorname{Pic}(X)$
in $H^2(X,\C)$.

Set $\tilde{\param}_{\R}:=\R\otimes_{\Z}\tilde{\param}_{\Z}$. There is a finite group $W$ acting on
$\tilde{\param}_{\R}$ as a reflection group, such that the movable cone $C$ of $X$ (that does not
depend on the choice of a resolution) is a fundamental chamber for $W$. We can find hyperplanes
$H_1,\ldots,H_k$ that are either walls of $C$ or pass through the interior of $C$ partitioning $C$
into the union of cones $C_1,\ldots,C_m$ such that the possible conical symplectic resolutions of
$X_0$ are in one-to-one correspondence with $C_1,\ldots,C_m$ in such a way that the cone corresponding
to a resolution $X'$ is its  ample cone.  Set
\begin{equation}\label{eq:singular} \mathcal{H}_{\C}:=\bigcup_{1\leqslant i\leqslant k,w\in W}w(\C\otimes_{\R}H_i).\end{equation}
Then $X_\lambda$ is affine if and only if $\lambda\not\in \mathcal{H}_{\C}$.
The results in this paragraph are due to Namikawa, \cite{Namikawa13}, see
also \cite[Section 2.3]{BPW} for an exposition.

\begin{defi}\label{defi:chamb_termin}
We say that an element $\theta\in \tilde{\param}_{\Q}$ is {\it generic} if it does not lie in $\mathcal{H}_\C$.
The cones $wC_i$ will be called {\it chambers}.
\end{defi}

For $\theta\not\in \mathcal{H}_{\C}$, we will write $X^\theta$ for the resolution corresponding to the ample cone
containing $W\theta\cap C$. Further, if $w\theta\in C$, we will choose a different identification
of $H^2(X^\theta,\C)$ with $\tilde{\param}$, one twisted by $w$ (so that the ample cone actually
contains $\theta$).

\subsection{Quantizations}
We start by introducing the various versions of quantizations that we are going to consider.

Let $X$ be a  Poisson scheme (of finite type over $\C$) and $d\in \Z_{>0}$. By a formal quantization  of $X$
(relative to $d$) we mean a sheaf of $\C[[\hbar]]$-algebras $\mathcal{D}_\hbar$ in Zariski topology
on $X$ together with an isomorphism $\iota:\mathcal{D}_\hbar/(\hbar)\xrightarrow{\sim}\mathcal{O}_X$
of sheaves of algebras that satisfy the following conditions.

\begin{itemize}
\item $\mathcal{D}_\hbar$ is flat over $\C[[\hbar]]$.
\item The $\hbar$-adic filtration on $\mathcal{D}_\hbar$ is complete and separated.
\item $[\mathcal{D}_\hbar,\mathcal{D}_\hbar]\subset \hbar^d \mathcal{D}_\hbar$. This gives
a Poisson bracket on $\mathcal{D}_\hbar/(\hbar)$.
\item $\iota$ is a Poisson isomorphism.
\end{itemize}

It is a standard fact that to give such a quantization in the case when $X$ is affine is the same
thing as to give a single algebra that is a formal quantization of $\C[X]$.

Now suppose that $\C^\times$ acts on $X$ in such a way that $t\in \C^\times$ rescales the Poisson bracket
on $\mathcal{O}_X$ by $t^d$. Then we can speak about graded formal quantization. By definition, this
is a formal quantization $(\mathcal{D}_\hbar,\iota)$ such that $\mathcal{D}_\hbar$ is equipped with a
$\C^\times$-action by algebra isomorphisms with $t.\hbar=t\hbar$ and such that $\iota$ is $\C^\times$-equivariant.

Let us now introduce the notion of a filtered quantization. Suppose that $\C^\times$ acts on $X$
as in the previous paragraph. Assume, in addition, that every point in $X$ has a $\C^\times$-stable
open affine neighborhood. This is the case when  $X$ is quasi-projective or by a theorem of
Sumihiro, \cite{Sumihiro}, when $X$ is normal. By the conical topology on $X$ we mean the
topology, where ``open'' means Zariski open and
$\C^\times$-stable. Note that $\mathcal{O}_X$ is a sheaf of graded algebras in the conical
topology.  By a filtered quantization of $\mathcal{O}_X$ we mean a pair $(\mathcal{D},\iota)$,
where $\mathcal{D}$ is a sheaf of $\Z$-filtered algebras (the filtration is ascending) in the conical topology on $X$
and $\iota$ is an isomorphism $\gr\mathcal{D}\xrightarrow{\sim}\mathcal{O}_X$
of sheaves of graded algebras. These data are supposed to satisfy the following axioms.
\begin{itemize}
\item The topology on $\mathcal{D}$ induced by the filtration is complete and separated.
\item If $\mathcal{D}_{\leqslant i}$ denotes the $i$th filtration component,
then $[\mathcal{D}_{\leqslant i},\mathcal{D}_{\leqslant j}]\subset \mathcal{D}_{\leqslant i+j-d}$
for all $i,j\in \Z$.
\item The isomorphism $\iota:\gr\mathcal{D}\xrightarrow{\sim} \mathcal{O}_X$ is Poisson.
\end{itemize}

Let us explain a connection between graded formal and filtered quantizations. Let $\mathcal{D}_\hbar$
be a graded formal quantization. Then we can consider the subsheaf of $\C^\times$-finite section
$\mathcal{D}_{\hbar,fin}$ of $\mathcal{D}_\hbar$ restricted to conical topology. Then
$\mathcal{D}_{\hbar,fin}/(\hbar-1)$ is a filtered quantization. Conversely,
let $\mathcal{D}$ be a filtered quantization. Then we can consider the Rees sheaf
$R_\hbar(\mathcal{D}):=\bigoplus_{i}\mathcal{D}_{\leqslant i}\hbar^i$. The $\hbar$-adic
completion $R_\hbar(\mathcal{D})$ uniquely extends to a sheaf in the Zariski topology that is a graded formal
quantization. It is easy to see that these two procedures give mutually inverse bijections
between the set of filtered quantizations and the set of graded formal quantizations.

Now let us discuss the classification of quantizations obtained in \cite{BK,quant}.
Let $\mathcal{Q}_{\hbar}(X)$ denote the set of isomorphism classes of formal quantizations of
$X$. Bezrukavnikov and Kaledin, \cite[Section 4]{BK}, defined a natural (in particular, compatible with pull-backs
under open embeddings) map $\mathsf{Per}: \mathcal{Q}_{\hbar}(X)\rightarrow H^2(X,\C)[[\hbar]]$.
They have shown in \cite[Theorem 1.8]{BK} that if $H^i(X,\mathcal{O}_X)=0$ for $i=1,2$, then this map is an isomorphism.
Now let $\mathcal{Q}_{\hbar,\C^\times}(X)$ denote the set of isomorphism classes of graded formal quantizations.
It was shown in \cite[Section 2.3]{quant} that the composition
$\mathcal{Q}_{\hbar,\C^\times}(X)\rightarrow \mathcal{Q}_\hbar(X)\rightarrow H^2(X,\C)[[\hbar]]$
restricts to  a bijection $\mathcal{Q}_{\hbar,\C^\times}(X)\xrightarrow{\sim}H^2(X,\C)$.
We conclude that the filtered quantizations of $X$ are parameterized by $H^2(X,\C)$.
For $\lambda\in H^2(X,\C)$, we write $\mathcal{D}_\lambda$ for the corresponding filtered
quantization.

The definitions of a quantization admit several ramifications. First, instead of $X$
we can consider a Poisson scheme $\widehat{X}$ over a vector space $\param$ and talk about
quantizations of $\widehat{X}/\param$ that are now required to be sheaves of $\C[\param]$-algebras. When we
speak about graded quantizations, we will always assume that $\C^\times$ acts on
$\param$ by $t.p=t^{-d}p$. Also we can talk about quantizations of formal schemes.
The classification results quoted in the previous paragraph still hold for formal quantizations.

Finally, let us discuss quantizations of a conical symplectic resolution $X=X^\theta$
(see \cite[Section 3]{BPW}). The  equalities $H^i(X,\mathcal{O}_X)=0$ hold so the filtered quantizations are classified
by $\tilde{\param}=H^2(X,\C)$. The quantization corresponding to $\lambda\in \tilde{\param}$ will be
denoted by $\A_\lambda^\theta$. Moreover,  recall that we have the universal conical deformation
$X_{\tilde{\param}}$ of $X$. It was shown in \cite[Section 6.2]{BK} that $X_{\tilde{\param}}$
admits a canonical quantization to be denoted by $\A^{\theta}_{\tilde{\param}}$.
It satisfies the following property: its specialization to $\lambda\in \tilde{\param}$
coincides with $\A^\theta_\lambda$.

Since $H^i(X,\mathcal{O}_X)=0$ for all $i>0$, we see that $\A_\lambda:=\Gamma(\A_\lambda^\theta)$
is a quantization of $\C[X]$, while $H^i(X,\mathcal{A}_\lambda^\theta)=0$. It was shown in
\cite[Section 3.3]{BPW} that $\A_\lambda$ is independent of the choice of $\theta$.

Note that $\A^\theta_{-\lambda}$ is identified with $(\A^\theta_\lambda)^{opp}$, \cite[Section 2.3]{quant},
and hence $\A_\lambda\cong \A_{-\lambda}^{opp}$.
Also we have $\A_\lambda\cong \A_{w\lambda}$ for all $\lambda\in \tilde{\param},w\in W$,
see \cite[Proposition 3.10]{BPW}.

\subsection{Example: quiver varieties}\label{SS_quiver}
Let us recall a special class of symplectic resolutions that will be of importance for us later.
This class is the Nakajima quiver varieties.

Let $Q$ be a quiver (=oriented graph, we allow loops and multiple edges). We can formally represent $Q$ as a quadruple
$(Q_0,Q_1,t,h)$, where $Q_0$ is a finite set of vertices, $Q_1$ is a finite set of arrows,
$t,h:Q_1\rightarrow Q_0$ are maps that to an arrow $a$ assign its tail and head. In this paper we
are interested in the case when $Q$ is of affine type, i.e., $Q$ is an extended quiver of type $A,D,E$.

Pick vectors $v,w\in \ZZ_{\geqslant 0}^{Q_0}$ and vector spaces $V_i,W_i$ with
$\dim V_i=v_i, \dim W_i=w_i$. Consider the (co)framed representation space
$$R=R(v,w):=\bigoplus_{a\in Q_1}\Hom(V_{t(a)},V_{h(a)})\oplus \bigoplus_{i\in Q_0} \Hom(V_i,W_i).$$
We will also consider its cotangent bundle $T^*R=R\oplus R^*$, this is a symplectic vector space that can be identified with
$$\bigoplus_{a\in Q_1}\left(\Hom(V_{t(a)},V_{h(a)})\oplus \Hom(V_{h(a)}, V_{t(a)})\right)\oplus \bigoplus_{i\in Q_0} \left(\Hom(V_i,W_i)\oplus \Hom(W_i,V_i)\right).$$
The group $G:=\prod_{k\in Q_0}\GL(V_k)$ naturally acts on $T^*R$ and this action is Hamiltonian. Its moment map $\mu:T^*R\rightarrow
\g^*$ is dual to
$x\mapsto x_R:\g\rightarrow \C[T^*R]$, where $x_R$ stands for the vector field on $R$ induced by $x\in \g$.

Fix a stability condition $\theta\in \Z^{Q_0}$ that is thought as a character of $G$ via $\theta((g_k)_{k\in Q_0})=\prod_{k\in Q_0}\det(g_k)^{\theta_k}$. Then, by definition, the quiver variety $\M^\theta(v,w)$ is the GIT Hamiltonian reduction $\mu^{-1}(0)^{\theta-ss}\quo G$. We are interested in two extreme cases: when $\theta$ is generic (and so $\M^\theta(v,w)$ is smooth and symplectic) and when $\theta=0$ (and so $\M^\theta(v,w)$ is affine). We will write $\M(v,w)$ for $\Spec(\C[\M^\theta(v,w)])$, this is an affine variety independent of $\theta$ and a natural projective morphism $\rho:\M^\theta(v,w)\rightarrow \M(v,w)$ is a resolution of singularities, see, for example,
\cite[Section 2.1]{BL}.
Note also that we have compatible $\C^\times$-actions on $\M^\theta(v,w),\M(v,w)$ induced from the action on $T^*R$
given by $t.(r,\alpha):=(t^{-1}r,t^{-1}\alpha), r\in R, \alpha\in R^*$. So $\M^\theta(v,w)\rightarrow \M(v,w)$
is a conical symplectic resolution. Note that if $\mu$ is flat, then $\M(v,w)\xrightarrow{\sim}\M^0(v,w)$.
There is a combinatorial necessary and sufficient condition on $\mu$ to be flat due to Crawley-Boevey, \cite{CB},
but we do not need that.

Now let us proceed to the quantum setting. We will work with quantizations of $\M^\theta(v,w),\M(v,w)$. Consider
the algebra $D(R)$ of differential operators on $R$. The group $G$ naturally acts on $D(R)$ with a quantum
comoment map $\Phi:\g\rightarrow D(R), x\mapsto x_R$. We can consider the quantum Hamiltonian reduction
$\A^0_\lambda(v,w)=[D(R)/D(R)\{x_R-\langle \lambda,x\rangle| x\in \g\}]^G$. It is a quantization of $\M^0(v,w)=\M(v,w)$
when the moment map $\mu$ is flat. We can also define a quantization $\A^\theta_\lambda(v,w)$ of $\M^\theta(v,w)$
 by quantum Hamiltonian reduction. Namely, we can microlocalize $D(R)$ to a sheaf in  the conical topology
so that we can consider the restriction of $D(R)$ to the $(T^*R)^{\theta-ss}$, let $\mathcal{D}^{\theta-ss}$ denote the restriction. Let $\pi$ stand for the quotient morphism $\mu^{-1}(0)^{\theta-ss}\rightarrow\mu^{-1}(0)^{\theta-ss}/G=\M^\theta(v,w)$.
Let us notice that $\mathcal{D}^{\theta-ss}/\mathcal{D}^{\theta-ss}\{x_R-\langle \lambda,x\rangle| x\in \g\}$
is scheme-theoretically supported on $\mu^{-1}(0)^{\theta-ss}$ and so can be regarded as a sheaf in
conical topology on that variety. Set
$$\A^\theta_{\lambda}(v,w):=[\pi_*\left(\mathcal{D}^{\theta-ss}/\mathcal{D}^{\theta-ss}\{x_R-\langle \lambda,x\rangle\}\right)]^G,$$
this is a quantization of $\M^\theta(v,w)$.
We note that the period of $\A^\theta_\lambda(v,w)$ equals to the cohomology class in $H^2(\M^\theta(v,w),\C)$
defined by $\lambda$
up to a shift by a fixed element in $H^2(\M^\theta(v,w),\C)$.
We will write $\A_\lambda(v,w)$ for $\Gamma(\A^\theta_\lambda(v,w))$.
We have $\A_\lambda(v,w)=\A_\lambda^0(v,w)$ for a  Zariski generic $\lambda$, see \cite[Section 2.2]{BL}.

Below we will need a standard and well-known result about symplectic leaves in quiver varieties.

\begin{Lem}\label{Lem:leaves}
Let $v'\leqslant v$ be a root of $Q$. Pick a Zariski generic parameter $\lambda$ with $v'\cdot \lambda=0$.
If the variety $\M_\lambda(v,w)$ has a single symplectic leaf that is a point, then $v'$ is a real root.
\end{Lem}
\begin{proof}
Consider a representation $r$ of $\bar{Q}^w$ (the notation from \cite[Section 2.1]{BL}) lying in
$\mu^{-1}(\lambda)$ and having closed orbit. Let $H$ denote the stabilizer of $r$ in $G$ and $U$
be the symplectic part of the slice representation in $X$. Then the point corresponding to
$x$ in the quotient $\mu^{-1}(\lambda)\quo G$ is a symplectic leaf if and only if $U^H=\{0\}$.
Recall that the space $U$ is recovered as follows. We decompose $r$ into the sum of the irreducible
$\bar{Q}^w$-modules: $r=r^0\oplus (r^1)^{\oplus n_1}\oplus\ldots\oplus (r_k)^{\oplus n_k}$,
where $r_0$ is an irreducible representation with dimension vector of the form $(v^0,1)$ and
$r_i$ are pairwise non-isomorphic irreducible representations of $\bar{Q}$ of dimension vector
$v^i$. By the choice of  $\lambda$, we see that all $v^i, i>0,$ are proportional to $v'$.
We note that $U$ is the representation of another quiver that has  vertices
$1,\ldots,k,$ with $1-\frac{1}{2}(v^i,v^i)$ loops at the vertex $i$. If $U^H=\{0\}$, then there are no
loops and so each $v^i$ is a real root and hence $v'$ is  a real root.
\end{proof}


\subsection{Harish-Chandra bimodules}
Now let $\A$ be a $\Z$-filtered algebra with a complete and separated filtration $\A=\bigcup_{i\in \Z}\A_{\leqslant i}$.
We assume that $[\A_{\leqslant i},\A_{\leqslant j}]\subset \A_{\leqslant i+j-d}$ for some fixed
$d\in \Z_{>0}$ and that $\gr\A$ is finitely generated.

By a Harish-Chandra $\A$-bimodule we mean a bimodule $\B$ that can be equipped with a complete and separated
filtration $\B=\bigcup_{i\in \Z}\B_{\leqslant i}$ such that
\begin{itemize}
\item The filtration is compatible with the filtration on $\A$.
\item $[\A_{\leqslant i}, \B_{\leqslant j}]\subset \B_{\leqslant i+j-d}$.
\item $\gr\B$ is a finitely generated $\gr\A$-module.
\end{itemize}

For a HC bimodule $\B$ we can define its associated variety $\VA(\B)$
inside of the reduced scheme associated to $\gr\A$ in the usual way.
Note that $\VA(\B)$ is a Poisson subvariety.

When $\A^1,\A^2$ are quotients of $\A$ we can speak about HC $\A^1$-$\A^2$-bimodules.

The following lemma is proved analogously to \cite[Proposition 3.8]{rouq_der}.

\begin{Lem}\label{Lem:HC_TorExt}
Let $\B^1,\B^2$ be HC $\A$-bimodules. Then $\operatorname{Tor}^{\A}_i(\B^1,\B^2),
\operatorname{Ext}^i_{\A}(\B^1,\B^2), \operatorname{Ext}^i_{\A^{opp}}(\B^1,\B^2)$ are HC.
\end{Lem}

Now let us give an example of a HC bimodule. Let $\A=\A_{\tilde{\param}}$. Pick $\chi\in \tilde{\param}_{\Z}$.
Consider a line  bundle $\mathcal{O}(\chi)$ on $X^\theta_{\tilde{\param}}$ (recall
that $\tilde{\param}_{\Z}$ is the image of $\operatorname{Pic}(X)$ in $H^2(X,\C)$,
for $\mathcal{O}(\chi)$ we take any line bundle corresponding to a lift of $\chi$
to $\operatorname{Pic}(X)$). Since $H^i(X^\theta_{\tilde{\param}},\mathcal{O}_{X^\theta_{\tilde{\param}}})=0$
for $i>0$, we see that $\mathcal{O}(\chi)$ admits a unique quantization $\A^\theta_{\tilde{\param},\chi}$
to a $\A^\theta_{\tilde{\param}}$-bimodule, where for $a\in \tilde{\param}^*$
we have $[a,m]=\langle\chi,a\rangle m$ for any local section $m$ of $\B$,
see \cite[Section 6.3]{BPW} for details.

Now pick an affine subspace $\param\subset \tilde{\param}$. Set $\A^\theta_{\param,\chi}:=
\A^\theta_{\tilde{\param},\chi}\otimes_{\C[\tilde{\param}]}\C[\param]$. Then
the global section bimodule $\A^{(\theta)}_{\param,\chi}:=\Gamma(\A^\theta_{\param,\chi})$
is HC over $\A_{\tilde{\param}}$.

Under some conditions, the bimodule $\A^{(\theta)}_{\param,\chi}$ is independent of
$\theta$.

\begin{Lem}\label{Lem:bimod_indep}
Suppose that the vector subspace of $\tilde{\param}$ associated to
$\param$ is not contained in the singular locus $\mathcal{H}_{\C}$. Then
$\A^{(\theta)}_{\param,\chi}$ is independent of $\theta$.
\end{Lem}

The proof was given in \cite[Proposition 5.5]{BL}. The following lemma is \cite[Proposition 6.26]{BPW}.

\begin{Lem}\label{Lem:spec}
Suppose that $H^1(X,\mathcal{O}(\chi))=0$. Then the specialization of $\A^{(\theta)}_{\param,\chi}$
to $\lambda\in \param$ coincides with $\A^{(\theta)}_{\lambda,\chi}$.
\end{Lem}

Now let us discuss $\widetilde{\param}$-supports of HC bimodules. Let $\B$ be a HC $\A_{\tilde{\param}}$-bimodule.
By its right $\tilde{\param}$-support we mean the set $\Supp^r_{\tilde{\param}}(\B)$ consisting of all
$\lambda\in \tilde{\param}$ such that $\B_\lambda:=\B\otimes_{\C[\tilde{\param}]}\C_\lambda$ is nonzero.

The following result was obtained in \cite[Proposition 2.6]{CWR}.

\begin{Lem}\label{Lem:supp_HC}
The subset $\Supp^r_{\tilde{\param}}(\B)\subset \tilde{\param}$ is Zariski closed. Its asymptotic
cone is $\Supp_{\tilde{\param}}(\gr\B)$, where the associated graded is taken with respect to
any good filtration.
\end{Lem}

For a HC $\A_{\tilde{\param}}$-bimodule $\B$, by $\VA_0(\B)$ we will denote $\VA(\B)\cap X_0$.

\subsection{Localization theorems}
Let $X=X^\theta$ be a conical symplectic resolution. Let $\lambda\in \tilde{\param}$
and $\A_\lambda^\theta$ be the corresponding filtered quantization
of $X$ with global sections $\A_\lambda$. We can consider the category
$\operatorname{Coh}(\A_\lambda^\theta)$ of all coherent sheaves of $\A_\lambda^\theta$-modules
and its derived category $D^b(\operatorname{Coh}(\A_\lambda^\theta))$.
We have the global section functor $\Gamma_\lambda=\Hom_{\A^\theta_\lambda}(\A_\lambda^\theta,\bullet):
\operatorname{Coh}(\A_\lambda^\theta)\rightarrow \A_\lambda\operatorname{-mod}$
and its left adjoint functor $\Loc_\lambda:=\A_\lambda^\theta\otimes_{\A_\lambda}\bullet:
\A_\lambda\operatorname{-mod}\rightarrow \A_\lambda^\theta\operatorname{-mod}$.

We say that abelian localization holds for $(\lambda,X)$ (or $(\lambda,\theta)$)
if $\Gamma_\lambda, \Loc_\lambda$ are mutually inverse equivalences.
The following result (proved in \cite[Section 5.3]{BPW}) gives a necessary and sufficient condition for
the abelian localization to hold. Let $\chi$ be an ample element in
$\tilde{\param}_{\Z}$.

\begin{Lem}\label{Lem:ab_loc_equiv}
The following two conditions on $\lambda\in \tilde{\param}$ are equivalent:
\begin{itemize}
\item Abelian localization holds for $(\lambda,\theta)$.
\item There is $n>0$ such that the bimodules $\A_{\lambda+mn\chi,n\chi}^{(\theta)}$
and $\A_{\tilde{\param},-n\chi}^{(\theta)}|_{\lambda+m(n+1)\chi}$ are mutually
inverse Morita equivalences for all $m\in \Z_{\geqslant 0}$.
\end{itemize}
\end{Lem}

Recall that a open subset $U\subset \tilde{\param}$ is said to be {\it asymptotically generic}
(a terminology from \cite{BL}) if the asymptotic cone of its complement is in $\mathcal{H}_{\C}$
(the union of singular hyperplanes).

\begin{Cor}\label{Cor:ab_loc}
Let $\chi\in \tilde{\param}_{\Z}$ be ample and
$m$ be such that the line bundle $\mathcal{O}_X(m\chi)$ has no higher cohomology.
Then there is an asymptotically generic open subset $U_{m,\chi}\subset \tilde{\param}$ such that
abelian localization holds for $(\lambda,\theta)$ provided $\lambda+nm\chi\in U_{m,\chi}$
for all $n\in \Z_{\geqslant 0}$.
\end{Cor}
\begin{proof}
The locus of $\lambda$, where the bimodules $\A_{\lambda,m\chi}^{(\theta)}$ and
$\A^{(\theta)}_{\param,-m\chi}|_{\lambda+m\chi}$ are inverse Morita equivalences
is Zariski open and asymptotically generic. This is proved using Lemma
\ref{Lem:supp_HC}, compare to the proof of (2) of \cite[Proposition 5.17]{BL}.
The claim of our corollary follows from Lemma \ref{Lem:ab_loc_equiv}.
\end{proof}

A weaker version of this result was obtained in \cite[Section 5.3]{BPW}.

We consider the derived functors $R\Gamma_\lambda:
D^b(\operatorname{Coh}(\A_\lambda^\theta))\rightarrow D^b(\A_\lambda\operatorname{-mod})$
and $L\Loc_\lambda: D^-(\A_\lambda\operatorname{-mod})\rightarrow
D^-(\operatorname{Coh}(\A_\lambda^\theta))$ (if $\A_\lambda$ has finite homological
dimension, then $L\Loc_\lambda$ restricts to a functor between bounded derived categories).

We say that {\it derived localization} holds for $(\lambda,X)$ (or $(\lambda,\theta)$)
if $R\Gamma_\lambda$ and $L\Loc_\lambda$ are mutually inverse equivalences.
In this case, $\A_\lambda$ has finite homological dimension. In all known examples,
it is known that the converse is also true, however, this fact is not proved in  general.

\subsection{Wall-crossing functors}
Let $\theta,\theta'$ be two generic elements of $H^2(X,\mathbb{Q})$ and $\lambda\in \tilde{\param}$.
Following \cite[Section 6.3]{BPW}, we are going to produce a derived equivalence $\WC_{\theta'\leftarrow \theta}:D^b(\Coh(\A_\lambda^\theta))\xrightarrow{\sim}
D^b(\Coh(\A_\lambda^{\theta'}))$ assuming abelian localization holds for $(\lambda,\theta)$
and derived localization holds for  $(\lambda,\theta')$. Then we set  $\WC_{\theta'\leftarrow \theta}:=L\Loc_{\lambda}^{\theta'}\circ \Gamma_\lambda^\theta$. Note that this functor
is right t-exact.

We can give a different realization of $\WC_{\theta'\leftarrow \theta}$. Namely, pick $\lambda'\in
\lambda+\tilde{\param}_{\Z}$ such that abelian localization holds for $(\lambda',\theta')$. We identify
$\Coh(\A_{\lambda}^\theta)$ with $\A_\lambda\operatorname{-mod}$ by means of $\Gamma_{\lambda}^\theta$
and $\Coh(\A_{\lambda}^{\theta'})$ with $\A_{\lambda'}\operatorname{-mod}$ by means of $\Gamma_{\lambda'}^{\theta'}(
\A^{\theta'}_{\lambda,\lambda'-\lambda}\otimes_{\A_\lambda^{\theta'}}\bullet)$. Under these identifications,
the functor $\WC_{\theta'\leftarrow \theta}$ becomes $\WC_{\lambda'\leftarrow \lambda}:=\A_{\lambda,\lambda'-\lambda}^{(\theta)}\otimes^L_{\A_\lambda}\bullet$, see \cite[Section 6.4]{BPW}.

We will need to study the functor $\WC_{\lambda, \lambda'-\lambda}$ as $\lambda'-\lambda$ is fixed
and $\lambda$ varies along a suitable affine subspace of $\tilde{\param}$. Namely, we take a face
$\Gamma$ of a chamber $C$ and consider the chamber $C'$ opposite to $C$ with respect $\Gamma$.
For example, if $\Gamma=\{0\}$, then $C'=-C$, while for $\Gamma$ of codimension $1$, we get
the unique chamber $C'$ sharing the face $\Gamma$ with $C$. Now pick a parameter $\lambda_0$
such that abelian localization holds for $(\lambda,\theta)$ with $\theta$ in the interior of $C$.
Let $\param_0$ be the vector subspace in $\tilde{\param}$ spanned by $\Gamma$ and
$\param:=\lambda_0+\param_0$. Further, fix $\chi\in \tilde{\param}_{\Z}$ such that abelian
localization holds for $\lambda'_0:=\lambda_0+\chi$ and $\theta'$, an element in the interior
of $C'$.

\begin{Prop}\label{Prop:WC_equiv}
Possibly after replacing $\lambda_0$ with an element $\lambda_0+\psi$ (and $\chi$ with $\chi-\psi$), where $\psi\in \tilde{\param}_{\Z}$ and abelian localization holds for $(\lambda_0+\psi, \theta)$, we have the following:  for a Zariski generic $\lambda\in \param$,
the functor $\A_{\param,\chi}^{(\theta')}|_{\lambda}\otimes^L_{\A_\lambda}\bullet$
is an equivalence $D^b(\A_\lambda\operatorname{-mod})\xrightarrow{\sim} D^b(\A_{\lambda+\chi}\operatorname{-mod})$.
\end{Prop}

In the proof we will need a connection between derived localization and global sections functors and homological duality
functors that we are now going to introduce. Assume that the algebra $\A_\lambda$ has finite
homological dimension. First, we have a functor
$D_\lambda: D^b(\A_\lambda\operatorname{-mod})\xrightarrow{\sim} D^b(\A_{-\lambda}\operatorname{-mod})^{opp}$
given by $R\Hom_{\A_\lambda}(\bullet,\A_\lambda)$ (here we use the identification $\A_\lambda^{opp}
\cong \A_{-\lambda}$). Second, for a generic element $\vartheta\in \tilde{\param}_{\Q}$ we have a
functor $D^\vartheta_\lambda:D^b(\Coh(\A_\lambda^{\vartheta}))\xrightarrow{\sim} D^b(\Coh(\A_{-\lambda}^\vartheta))^{opp}$
given by $R\mathcal{H}om_{\A_\lambda^\vartheta}(\bullet,\A_\lambda^\vartheta)$.

\begin{Lem}\label{Lem:D_Gamma}
We have $D_\lambda\cong R\Gamma_{-\lambda}^\vartheta\circ  D^\vartheta_\lambda \circ L\Loc_{\lambda}^\vartheta$.
\end{Lem}
\begin{proof}
Note that $R\Gamma_{-\lambda}^\vartheta\circ D^\vartheta_\lambda(\bullet)=R\Hom_{\A^\vartheta_\lambda}(\bullet, \A_\lambda^\vartheta)$. On the other hand, $R\Hom_{\A^\vartheta_\lambda}(\A_\lambda^\vartheta\otimes^L_{\A_\lambda}M,
\A_\lambda^\vartheta)=R\Hom_{\A_\lambda}(M, \A_\lambda^\vartheta)=R\Hom_{\A_\lambda}(M,\A_\lambda)$. Here
$R\Hom_{\A_\lambda}(M,\A_\lambda)\rightarrow R\Hom_{\A_\lambda}(M,\A_\lambda^\vartheta)$ is induced
by $\A_\lambda\rightarrow \A_\lambda^\vartheta$, it is an isomorphism for any $M$ because it is
an isomorphism for $M=\A_\lambda$.  This completes the proof.
\end{proof}

\begin{Cor}\label{Cor:der_loc}
Suppose that $\A_\lambda,\A_{-\lambda}$ have finite homological dimension. The functor $R\Gamma_{-\lambda}^\vartheta$ is
an equivalence if and only if $L\Loc_\lambda^\vartheta$ is. Equivalently, derived localization holds
for $(\lambda,\vartheta)$ if and only if it holds for $(-\lambda,\vartheta)$.
\end{Cor}

\begin{proof}[Proof of Proposition \ref{Prop:WC_equiv}]
It is enough to prove this claim for a Weil generic element of $\param$, compare to \cite[Section 5.3]{rouq_der}.
We only need to check that derived localization holds for $(\lambda,\theta')$, by Corollary \ref{Cor:der_loc},
this is equivalent to the claim that derived localization holds for $(-\lambda,\theta')$.
We note that (perhaps, after replacing $\lambda_0$ with $\lambda_0+\psi$ as in the statement
of the proposition) abelian localization holds for  $(-\lambda,\theta')$ (this follows from the choice of $\lambda$ (Weil generic) and Corollary \ref{Cor:ab_loc}).
The claim of the proposition follows.
\end{proof}

Now let us discuss a special class of wall-crossing functors, the long wall-crossing functors.

Pick a generic $\theta\in \tilde{\param}_{\Q}$ and $\lambda,\lambda^-\in \tilde{\param}$ subject to
the following conditions:
\begin{enumerate}
\item Abelian localization holds for $(\lambda,\theta),(-\lambda,-\theta),(\lambda^-,\theta)$.
\item $\lambda^--\lambda\in \tilde{\param}_{\Z}$.
\end{enumerate}
So we get the wall-crossing functor $\WC_{\lambda^-\leftarrow \lambda}:D^b(\A_\lambda\operatorname{-mod})
\rightarrow D^b(\A_{\lambda^-}\operatorname{-mod})$.

We want to study the behavior of this functor  on the subcategory
$D^b_{hol}(\A_\lambda)\subset D^b(\A_\lambda\operatorname{-mod})$ of all objects
with holonomic homology. Recall that we say that an $\A_\lambda$-module $M$ is
{\it holonomic} if  $\VA(M)$ intersects all leaves of $X_0$ at isotropic
subvarieties, equivalently, if $\pi^{-1}(\VA(M))$ is an isotropic subvariety
of $X$, see \cite[Section 5]{B_ineq}.

It is easy to see that $\VA(H^i(DM))\subset \VA(M)$ for all $i$. From here and  $D^2=\operatorname{id}$, it
follows that  $D$ restricts to an equivalence
$D^b_{hol}(\A_\lambda)\xrightarrow{\sim}D^b_{hol}(\A_{-\lambda})^{opp}$.

Similarly, we can define the full subcategory $D^b_{hol}(\A_\lambda^\vartheta)\subset
D^b(\Coh(\A_\lambda^\vartheta))$. It was checked in \cite[Section 4]{BL}, that
$D^\vartheta_\lambda[\frac{1}{2}\dim X]$ is a t-exact equivalence
$D^b_{hol}(\A_\lambda^\vartheta)\xrightarrow{\sim} D^b_{hol}(\A_{-\lambda}^\vartheta)$.

The functor $\WC_{\lambda^-\leftarrow \lambda}$ also restricts to
an equivalence $D^b_{hol}(\A_\lambda)\xrightarrow{\sim} D^b_{hol}(\A_{\lambda^-})$.

The following result was obtained (in a special case but the proof in the general
case is the same) in \cite[Section 4]{BL}.

\begin{Prop}\label{Prop:WC_D}
Under the assumptions (1),(2) above, there is a t-exact equivalence $D^b_{hol}(\A_{\lambda^-})\xrightarrow{\sim}
D^b_{hol}(\A_{-\lambda})^{opp}$ that intertwines $\WC_{\lambda^-\leftarrow \lambda}$
with $D$.
\end{Prop}

Here is a corollary of this proposition also proved in \cite[Section 4]{BL}.

\begin{Cor}\label{Cor:fin_dim_shift}
For $M\in \A_{\lambda}\operatorname{-mod}$, the following two conditions are equivalent:
\begin{enumerate}
\item $H_i(\WC_{\lambda^-\leftarrow \lambda}M)=0$ for $i<\frac{1}{2}\dim X$.
\item $\dim M<\infty$.
\end{enumerate}
\end{Cor}

\section{Perversity of wall-crossing}\label{S_perv}
\subsection{Main result}
Let $(\lambda^1,\theta^1),(\lambda^2,\theta^2)\in \tilde{\param}\times \tilde{\param}_{\Q}$,
where $\theta^1,\theta^2$ are generic,  be  such that abelian localization holds and
$\lambda^2-\lambda^1\in \tilde{\param}_{\Z}$. Let $C^1,C^2$ denote  classical
chambers of $\theta^1,\theta^2$, respectively. We assume that $C^1$ and $C^2$ are opposite
to each other with respect to their common face, say $\Gamma$. In other words, there is
an interval whose midpoint is generic in $\Gamma$, while the end points are generic
in $C^1,C^2$. We are going to prove that the functor $\WC_{\lambda^2\leftarrow \lambda^1}$
is a {\it perverse equivalence} in the sense of Chuang and Rouquier, \cite[Section 2.6]{Rouquier_ICM}.

Let us recall the general definition.
Let $\mathcal{T}^1,\mathcal{T}^2$ be  triangulated categories equipped with  $t$-structures
that are homologically finite (each object in $\mathcal{T}^i$ has only finitely many nonzero
homology groups). Let $\mathcal{C}^1,\mathcal{C}^2$ denote the hearts of $\mathcal{T}^1,\mathcal{T}^2$, respectively.

We  are going to define a perverse equivalence with respect to  filtrations
$\mathcal{C}^i=\mathcal{C}_0^i\supset \mathcal{C}_1^i
\supset\ldots \supset\mathcal{C}_k^i=\{0\}$ by Serre subcategories. By definition, this is a triangulated equivalence
$\mathcal{T}^1\rightarrow \mathcal{T}^2$ subject to the following
conditions:
\begin{itemize}
\item[(P1)] For any $j$, the equivalence $\mathcal{F}$ restricts to an equivalence
$\mathcal{T}^1_{\mathcal{C}_j^1}\rightarrow \mathcal{T}^2_{\mathcal{C}_j^2}$, where
we write $\mathcal{T}^i_{\mathcal{C}_j^i}, i=1,2,$ for the category of all objects
in $\mathcal{T}^i$ with homology (computed with respect to the t-structures of interest)
in $\mathcal{C}_j^i$.
\item[(P2)] For $M\in \Cat_j^1$, we have $H_\ell(\mathcal{F}M)=0$ for $\ell<j$
and $H_\ell(\mathcal{F}M)\in \Cat^2_{j+1}$ for $\ell>j$.
\item[(P3)] The  functor $M\mapsto H_j(\mathcal{F}M)$ induces an equivalence $\Cat^1_j/\Cat^1_{j+1}\xrightarrow{\sim}
\Cat^2_j/\Cat^2_{j+1}$  of abelian categories.
\end{itemize}

Now let $\param^1,\param^2\subset \tilde{\param}$ denote the affine subspaces
$\param^i:=\lambda^i+\operatorname{Span}_{\C}(\Gamma)$. We can shift the space
$\param^1$,  see Proposition \ref{Prop:WC_equiv}, such that the derived localization
holds for a Weil generic point of this space and the stability condition $\theta^2$.
We will produce the chains of two-sided ideals
$$\A_{\param^i}:=\I_{\param^i}^0\supset \I_{\param^i}^1\supset\ldots
\I_{\param^i}^q\supset \I_{\param^i}^{q+1}=\{0\},$$
where $q=\frac{1}{2}\dim X$,  having the following property:
\begin{itemize}
\item[(*)] For a Weil generic parameter $\widehat{\lambda}^i\in \param^i$,
the specialization $\I^j_{\widehat{\lambda}^i}$ is the minimal two-sided
ideal $\I\subset \A_{\widehat{\lambda}^i}$ such that $\operatorname{GK}$-$\dim\A_{\widehat{\lambda}^i}/\I<2j$.
\end{itemize}

(*) implies that $(\I^j_{\lambda^i})^2=\I^j_{\lambda^i}$ for a Weil generic $\lambda^i\in \param^i$
and hence also for a Zariski generic $\lambda^i$. Note that the ideal $\I^j_{\lambda^i}$
is well-defined for a Zariski generic $\lambda^i$ as in the proof of  \cite[Lemma 2.9]{Cher_supp}. We set $\Cat^i_j=(\A_{\lambda^i}/\I^j_{\lambda^i})\operatorname{-mod}$. This is a Serre
subcategory of $\Cat^i$.

\begin{Thm}\label{Thm:perv}
We assume that $X^0$ has conical slices (see Definition \ref{defi:con_slices} below).
Suppose that derived localization holds for a Weil generic point  $\widehat{\lambda}^1\in\param^1$ and $\theta^2$.
Pick $\chi$ in the chamber of $\theta_2$ such that $H^1(X^{\theta_2}, \mathcal{O}(\chi))=0$.
For a Zariski generic $\lambda^1\in\param^1$ and $\chi\in \tilde{\param}_{\Z}$ such that
$\param^2=\param^1+\chi$ and abelian localization holds for $(\lambda^i,\theta^i)$, where
$\lambda^2=\lambda^1+\chi$,
the functor $\WC_{\lambda^1\rightarrow \lambda^2}$ is a perverse equivalence with respect to the filtrations
$\Cat_j^i\subset \Cat^i, i=1,2$.
\end{Thm}

\subsection{Slices}\label{SS_slices}
Here we are going to impose an additional assumption. Pick a point $x\in X^0$.
Let $\mathcal{L}=\mathcal{L}_x$ denote the symplectic leaf through $x$. Consider the
completion $\C[X^0]^{\wedge_x}$ of $\C[X^0]$. Then we can embed $\C[\mathcal{L}_x]^{\wedge_x}$
into $\C[X^0]^{\wedge_x}$ and this embedding is unique up to a twist with a Hamiltonian
automorphism of $\C[X^0]^{\wedge_x}$, see \cite[Section 3]{Kaledin}. Moreover, $\C[X^0]^{\wedge_x}$ splits into
the completed tensor product $\C[X^0]^{\wedge_x}=\C[\mathcal{L}]^{\wedge_x}\otimes \underline{\widehat{A}}_x$,
where $\underline{\widehat{A}}_x$ is the centralizer of $\C[\mathcal{L}]^{\wedge_x}$ in $\C[X^0]^{\wedge_x}$.
Let $\underline{\widehat{X}}^0(=\underline{\widehat{X}}^0_x)$ denote the formal spectrum of $\underline{\widehat{A}}_x$,
we can view $\underline{\widehat{X}}_0$ as a formal subscheme of $X^0$ so that $X^{0\wedge_x}=
\mathcal{L}^{\wedge_x}\times \underline{\widehat{X}}_0$. We call $\underline{\widehat{X}}^0$ a slice to $x$ in
$X^0$.

\begin{defi}\label{defi:con_slices}
We say that $\underline{\widehat{X}}^0$ is {\it conical} if there is a pro-rational $\C^\times$-action on
$\underline{\widehat{A}}$ that rescales the Poisson bracket on $\underline{\widehat{A}}$ by $t^{-d}$ and whose weights
on the maximal ideal of $\underline{\widehat{A}}$ are positive. Further, we say that $X^0$
{\it has conical slices} if $\underline{\widehat{X}}^0$ is conical for all $x\in X^0$. This holds for all
examples of $X^0$  that we know.
\end{defi}

Let $\underline{A}$ denote the $\C^\times$-finite part of $\underline{\widehat{A}}$. This is a Poisson subalgebra.
Then $\underline{X}^0:=\operatorname{Spec}(\underline{A})$ is a Poisson variety with a contracting action
of $\C^\times$. Moreover, $\underline{X}^0$ admits a symplectic resolution: the formal neighborhood of
the zero fiber in the resolution of $\underline{X}^0$ coincides with the formal neighborhood of
$\rho^{-1}(x)$. So $\underline{\tilde{\param}}:=H^2(\underline{X},\C)=H^2(\rho^{-1}(x),\C)$.

Now let $\A$ be a quantization of $\C[X^0]$. We are going to produce a slice quantization of $\C[\underline{X}^0]$,
a construction that first appeared in \cite{HC} with some refinements given in \cite{W_prim}. In what follows
we assume that $d$ is even (we can always replace $d$ with its multiple by replacing the
contracting torus with its cover). Let $V:=T_x\mathcal{L}$, this is
a symplectic vector space (let $\omega_V$ denote the form).
Consider the homogenized Weyl algebra $\mathbb{A}_\hbar(V)$ with the
relations $uv-vu=\hbar^d \omega_V(u,v)$. We can consider the completion $\mathbb{A}_\hbar(V)^{\wedge_0}$
of $\mathbb{A}_\hbar(V)$ at $0\in V$. Note that this is an algebra flat over $\C[[\hbar]]$
and $\mathbb{A}_\hbar(V)^{\wedge_0}/(\hbar)=\C[\mathcal{L}]^{\wedge_x}$.

Now consider the Rees algebra $\A_\hbar$ of $\A$ and its completion $\A_\hbar^{\wedge_x}$.
We can lift the embedding $\C[\mathcal{L}]^{\wedge_x}\hookrightarrow \C[X^0]^{\wedge_x}$ to
an embedding $\mathbb{A}_\hbar(V)^{\wedge_0}\hookrightarrow \A_\hbar^{\wedge_x}$ (that is unique
up to a twist with an automorphism of the form $\exp(\hbar^{1-d}f)$ for $f\in \A_\hbar^{\wedge_x}$),
see \cite[Section 2.1]{W_prim}.
Moreover, the centralizer $\underline{\widehat{\A}}_\hbar$ of $\mathbb{A}_\hbar(V)^{\wedge_0}$
in $\A_\hbar^{\wedge_x}$ satisfies $\underline{\widehat{\A}}_\hbar/(\hbar)=\underline{\widehat{A}}$
so that we have $\A_\hbar^{\wedge_x}=\mathbb{A}_\hbar(V)^{\wedge_0}\widehat{\otimes}_{\C[[\hbar]]}
\underline{\widehat{\A}}_\hbar$.

Assume now that $\A=\A_\lambda$, where $\lambda\in \tilde{\param}$.

\begin{Lem}\label{Lem:slice_quant}
The following is true:
\begin{enumerate}
\item The action of $\C^\times$ on $\underline{\widehat{A}}$ lifts to a pro-rational $\C^\times$-action
on $\underline{\widehat{\A}}_\hbar$ by algebra automorphisms with $\hbar$ of degree $1$.
\item Let $\underline{\A}_\hbar$ denote the $\C^\times$-finite part of $\underline{\widehat{\A}}_\hbar$,
then $\underline{\A}:=\underline{\A}_\hbar/(\hbar-1)$ is the algebra of global
sections of the filtered quantization of $\underline{X}$, whose period is the pull-back
of $\lambda$ to $\rho^{-1}(x)$.
\end{enumerate}
\end{Lem}
\begin{proof}
Let us prove (1). It is enough to lift the $\C^\times$-action on $\underline{\widehat{A}}\widehat{\otimes}
\C[[V]]$ to $\underline{\widehat{\A}}_\hbar\widehat{\otimes}_{\C[[\hbar]]}\mathbb{A}_\hbar(V)^{\wedge_0}$
(by changing the embedding of $V$ into the latter tensor product we may achieve that $V$ is
$\C^\times$-stable so that the $\C^\times$-action will restrict to $\underline{\widehat{\A}}_\hbar$).

The action of $\C^\times$ on $\A_\hbar$ gives rise to the Euler derivation that we denote by $\mathsf{eu}$.
The derivation extends to the completion $\A_\hbar^{\wedge_x}$ that we have identified with $\underline{\widehat{\A}}_\hbar\widehat{\otimes}_{\C[[\hbar]]}\mathbb{A}_\hbar(V)^{\wedge_0}$.
Now on $\underline{\widehat{A}}\widehat{\otimes}\C[[V]]$ we have two derivations $\mathsf{eu}$ and $\underline{\mathsf{eu}}$,
the latter comes from the $\C^\times$-action on $\underline{\widehat{A}}\widehat{\otimes}\C[[V]]$.
The difference $\delta:=\mathsf{eu}-\underline{\mathsf{eu}}$ is a Poisson derivation
of $\underline{\widehat{A}}\widehat{\otimes}\C[[V]]$. It is enough to show that it is Hamiltonian,
then we can lift $\underline{\mathsf{eu}}$ to a derivation  $\underline{\mathsf{eu}}$
of $\underline{\widehat{\A}}_\hbar\widehat{\otimes}_{\C[[\hbar]]}\mathbb{A}_\hbar(V)^{\wedge_0}$,
which is easily seen to integrate to a $\C^\times$-action.

To prove that $\delta$ is Hamiltonian, we note that both $\mathsf{eu}$ and $\underline{\mathsf{eu}}$
lift to $\underline{\widehat{X}}$. So $\delta$ also lifts to a  symplectic vector field on $\underline{\widehat{X}}$.
On the other hand, by a result of Kaledin, \cite[Corollary 1.5]{Kaledin_survey}, $H^1(\pi^{-1}(x),\C)=0$. It follows that
$\underline{\widehat{X}}$ has no 1st de Rham cohomology. This shows that $\delta$ is Hamiltonian
and finishes the proof of (1).

Let us prove (2). We can still form the completion $(\A_\hbar^{\theta})^{\wedge_{\pi^{-1}(x)}}$
that will be a formal quantization of the formal neighborhood $X^{\wedge_{\pi^{-1}(x)}}$ of
$\pi^{-1}(x)$ in $X$. The notion of  period still makes sense. It follows from the construction in \cite{BK}
that the period of $(\A_\hbar^{\theta})^{\wedge_x}$ is the pull-back $\underline{\lambda}$ of $\lambda$
to $\pi^{-1}(x)$. On the other hand, the centralizer of $\mathbb{A}_\hbar(V)^{\wedge_0}$
in $(\A_\hbar^{\theta})^{\wedge_x}$ (to be denoted by $\underline{\widehat{\A}}^{\theta}_\hbar$)
is a quantization of the slice in  $X^{\wedge_{\pi^{-1}(x)}}$. Note that the global sections of
$\underline{\widehat{\A}}^{\theta}_\hbar$ are $\underline{\widehat{\A}}_\hbar$. By the proof of (1), $\C^\times$ acts
by automorphisms of $(\A_\hbar^{\theta})^{\wedge_{\pi^{-1}(x)}}$ preserving $V$ so that, modulo $\hbar$,
we get the contracting action on $X^{\wedge_{\pi^{-1}(x)}}$. So $\C^\times$ acts on
$\underline{\widehat{\A}}^\theta_\hbar$. From here we get a filtered quantization $\underline{\A}^\theta$
of $\underline{X}$ with period $\underline{\lambda}$. By the construction, its
global sections are $\underline{\A}$. This finishes the proof of (2).
\end{proof}

\begin{Rem}
We can also consider slices for the varieties $X^0_{\tilde{\param}},X^\theta_{\tilde{\param}}$
and the algebra $\A_{\tilde{\param}}$. For the same reasons as before, we get the deformations
$\underline{X}^0_{\tilde{\param}},\underline{X}^\theta_{\tilde{\param}}$ of $\underline{X}^0,\underline{X}$.
Further, we get the quantization $\underline{\A}^\theta_{\tilde{\param}}$ of $\underline{X}_{\tilde{\param}}$
and its algebra of global sections $\underline{\A}_{\tilde{\param}}$. Part (2) of Lemma
\ref{Lem:slice_quant} shows that $\underline{\A}^\theta_{\tilde{\param}}$ is obtained from
the canonical quantization of $\underline{X}^\theta_{\tilde{\underline{\param}}}$ under the
pull-back with respect to the natural map map $\tilde{\param}\rightarrow \tilde{\underline{\param}}$.
Recall that we write $\tilde{\underline{\param}}$ for $H^2(\underline{X},\C)$.
\end{Rem}

\subsection{Restriction functors for HC bimodules}
Now we are going to define restriction functors between the categories of HC bimodules.
Namely, we pick a point $x\in X^0$. This allows us to define the slice algebras
$\underline{\A}_\lambda$ for $\A_\lambda$ and $\underline{\A}_{\tilde{\param}}$ for
$\A_{\tilde{\param}}$. We are going to produce an exact functor $\bullet_{\dagger,x}:\HC(\A_{\tilde{\param}})
\rightarrow \HC(\underline{\A}_{\tilde{\param}})$.

Pick $\B\in \HC(\A_{\tilde{\param}})$. Choose a good filtration on $\B$ and form the corresponding
Rees $\A_{\tilde{\param},\hbar}$-bimodule $\B_\hbar$. This bimodule comes with the Euler derivation
$\mathsf{eu}$. The derivation extends to the completion $\B_\hbar^{\wedge_x}$ that is an $\A_{\tilde{\param},\hbar}^{\wedge_x}$-bimodule. Similarly to \cite[Section 3.3]{HC}, we
see that $\B_\hbar^{\wedge_x}$ decomposes into the product
$\mathbb{A}_\hbar(V)^{\wedge_0}\widehat{\otimes}_{\C[[\hbar]]}\underline{\widehat{\B}}_\hbar$,
where $\underline{\widehat{\B}}_\hbar$ stands for the centralizer of $V$ in $\B_\hbar^{\wedge_x}$.
Note that we can equip $\underline{\widehat{\B}}_\hbar$ with an Euler derivation that is compatible
with the derivation $\underline{\mathsf{eu}}$ on $\underline{\widehat{\A}}_{\tilde{\param},\hbar}$.
Namely, recall, Lemma \ref{Lem:slice_quant}, that there is an element $a\in\mathbb{A}_\hbar(V)^{\wedge_0}\widehat{\otimes}_{\C[[\hbar]]}\underline{\widehat{\A}}_{\tilde{\param},\hbar}$
such that $\mathsf{eu}-\underline{\mathsf{eu}}=\hbar^{-d}[a,\cdot]$. Now we can {\it define}
the derivation $\underline{\mathsf{eu}}$ of $\B^{\wedge_x}_\hbar$ as $\mathsf{eu}-\hbar^{-d}[a,\cdot]$.
It restricts to $\underline{\widehat{\B}}_\hbar$. Then we define the $\underline{\A}_{\tilde{\param},\hbar}$-subbimodule
$\underline{\B}_\hbar$ of $\underline{\widehat{\B}}_\hbar$ as the the $\underline{\mathsf{eu}}$-finite part of
$\underline{\widehat{\B}}_\hbar$. The bimodule $\underline{\B}_\hbar$ is gradeable and is finitely generated
over $\underline{\A}_\hbar$. We set $\B_{\dagger,x}:=\underline{\B}_\hbar/(\hbar-1)\underline{\B}_\hbar$.
The assignment $\B\mapsto \B_{\dagger,x}$ is indeed a functor, it follows easily from the construction
that this functor is exact.

The following two properties of $\bullet_{\dagger,x}$ are established as in \cite[Section 5.5]{BL}.

\begin{Lem}\label{Lem:dagger_assoc_var}
The associated variety $\VA(\B_{\dagger,x})$ is the unique
conical subvariety in $\underline{X}_{0,\tilde{\param}}$ such that
$\VA(\B_{\dagger,x})\times \mathcal{L}^{\wedge_x}=\VA(\B)^{\wedge_x}$.
\end{Lem}

\begin{Lem}\label{Lem:dagger_Tor_Ext}
The functor $\bullet_{\dagger,x}$ intertwines the Tor and the Ext functors. More precisely,
we get the following:
\begin{align*}
&\operatorname{Tor}^{\A_{\tilde{\param}}}_i(\B^1,\B^2)_{\dagger,x}\cong
\operatorname{Tor}^{\underline{\A}_{\tilde{\param}}}_i(\B^1_{\dagger,x}, \B^2_{\dagger,x}),\\
&\operatorname{Ext}_{\A_{\tilde{\param}}}^i(\B^1,\B^2)_{\dagger,x}\cong
\operatorname{Ext}_{\underline{\A}_{\tilde{\param}}}^i(\B^1_{\dagger,x}, \B^2_{\dagger,x}). \end{align*}
The latter equality holds for Ext's of left $\A_{\tilde{\param}}$-modules and of right modules.
\end{Lem}

Now pick a point $x\in X^0$ and let $\mathcal{L}$ denote the leaf through $x$.
Consider the full subcategory $\HC_{\overline{\mathcal{L}}}(\A_{\tilde{\param}})
\subset \HC(\A_{\tilde{\param}})$ consisting of all HC bimodules $\B$ such that
$\VA_0(\B)\subset \overline{\mathcal{L}}$. Also consider the full subcategory
$\HC_{fin}(\underline{\A}_{\tilde{\param}})\subset \HC(\underline{\A}_{\tilde{\param}})$
of all bimodules that have finite rank $\C[\tilde{\param}]$.
By Lemma \ref{Lem:dagger_assoc_var}, the functor $\bullet_{\dagger,x}$
restricts to $\HC_{\overline{\mathcal{L}}}(\A_{\tilde{\param}})\rightarrow
\HC_{fin}(\underline{\A}_{\tilde{\param}})$.

The following lemma is proved in the same way as a similar claim in \cite{BL},
see Proposition 5.21 in {\it loc.cit}.

\begin{Lem}\label{Lem:adj_functor}
The functor $\bullet_{\dagger,x}:\HC_{\overline{\mathcal{L}}}(\A_{\tilde{\param}})\rightarrow
\HC_{fin}(\underline{\A}_{\tilde{\param}})$ admits a right adjoint functor,
to be denoted by $\bullet^{\dagger,x}$.
\end{Lem}

Now let us study the behavior of $\bullet_{\dagger,x}$ on wall-crossing bimodules.

\begin{Prop}\label{Prop:WC_restr}
Suppose that $\lambda\in \tilde{\param}, \chi\in \tilde{\param}_{\Z}$, abelian localization
holds for $(\lambda+\chi,\theta)$ and that $H^1(X^\theta,\mathcal{O}(\chi))=0$. Then $(\A_{\lambda,\chi}^{(\theta)})_{\dagger,x}=\underline{\A}^{(\theta)}_{\lambda,\chi}$.
\end{Prop}
Here and below we write $\underline{\A}^{(\theta)}_{\lambda,\chi}$ for
$\underline{\A}_{\tilde{\param}}$-bimodule defined similarly to $\A^{(\theta)}_{\lambda,\chi}$.
\begin{proof}
From $H^1(X^\theta, \mathcal{O}(\chi))=0$ we deduce that $\gr\A_{\lambda,\chi}^{(\theta)}=\Gamma(\mathcal{O}(\chi))$.
By the formal function theorem, we have  that $\Gamma(\mathcal{O}(\chi))^{\wedge_x}$ coincides
with the global sections of $\mathcal{O}(\chi)^{\wedge_{\pi^{-1}(x)}}$. It follows that
\begin{equation}\label{eq:restr_eq1}
(\A^{(\theta)}_{\lambda,\chi,\hbar})^{\wedge_x}/\hbar (\A^{(\theta)}_{\lambda,\chi,\hbar})^{\wedge_x}=
\Gamma(\mathcal{O}(\chi)^{\wedge_{\pi^{-1}(x)}})\end{equation}
Now let $\underline{\mathcal{O}}(\chi)$ denote the line bundle on $\underline{X}$ obtained by restricting
$\mathcal{O}(\chi)$. From (\ref{eq:restr_eq1}) and the construction of the functor $\bullet_{\dagger,x}$,
we conclude that
\begin{equation}\label{eq:restr_eq2}
\gr\left((\A_{\lambda,\chi}^{(\theta)})_{\dagger,x}\right)=\Gamma(\underline{\mathcal{O}}(\chi)).
\end{equation}
On the other hand, we have a natural homomorphism
$$(\A^{(\theta)}_{\lambda,\chi,\hbar})^{\wedge_x}
\rightarrow \Gamma\left( (\A^{\theta}_{\lambda,\chi,\hbar})^{\wedge_{\pi^{-1}(x)}}\right).$$
Note that
$$\Gamma\left( (\A^{\theta}_{\lambda,\chi,\hbar})^{\wedge_{\pi^{-1}(x)}}\right)=\mathbb{A}_\hbar^{\wedge_0}
\widehat{\otimes}_{\C[[\hbar]]}\Gamma\left( (\underline{\A}^{\theta}_{\lambda,\chi,\hbar})^{\wedge_{\underline{\pi}^{-1}(0)}}\right).$$
This yields a filtered bimodule homomorphism $(\A^{(\theta)}_{\lambda,\chi})_{\dagger,x}\rightarrow
\underline{\A}^{(\theta)}_{\lambda,\chi}$. The corresponding homomorphism
of the associated graded bimodules intertwines the isomorphism $
\gr\left((\A^{(\theta)}_{\lambda,\chi})_{\dagger,x}\right)\xrightarrow{\sim}\Gamma(\underline{\mathcal{O}}(\chi))$
and the inclusion $\gr \underline{\A}^{(\theta)}_{\lambda,\chi}\hookrightarrow \Gamma(\underline{\mathcal{O}}(\chi))$.
It follows that $\gr\left((\A^{(\theta)}_{\lambda,\chi})_{\dagger,x}\right)\xrightarrow{\sim}
\gr \underline{\A}^{(\theta)}_{\lambda,\chi}$ and hence
$(\A^{(\theta)}_{\lambda,\chi})_{\dagger,x}\xrightarrow{\sim}\underline{\A}^{(\theta)}_{\lambda,\chi}$.
\end{proof}

\subsection{Proof of Theorem \ref{Thm:perv}}
The proof follows the strategy of \cite[Section 6]{rouq_der} using besides
Corollary \ref{Cor:fin_dim_shift} and Proposition \ref{Prop:WC_restr}.

First, let us produce the ideals $\I^j_{\param^i}, i=1,2$. We start with $\I^1_{\param^i}$.

\begin{Lem}\label{Lem:ideal_spec}
There is an ideal $\I_{\param^i}^1\subset \A_{\param^i}$  that specializes to the minimal ideal
of finite codimension for a
Weil generic parameter in $\param^i$ and such that $\A^{\param^i}/\I^{\param^i}_1$ is finitely
generated as a module over $\C[\param^i]$.
\end{Lem}
\begin{proof}
The proof repeats that of \cite[Lemma 5.1]{rouq_der} (using the fact that the algebra
$\A_\lambda$ has a unique minimal ideal of finite codimension, see \cite[Section 4.3]{B_ineq}, instead of appealing to
the category $\mathcal{O}$ as in \cite{rouq_der}).
\end{proof}

Now let us construct the  ideals $\I^j_{\param^i}$ for arbitrary $j$
as in \cite[Section 5.2]{rouq_der}.
We set $$\I^j_{\param^i}=\left( \bigcap_{\mathcal{L}} (\I^{1,\mathcal{L}}_{\param^i})^{\dagger,\mathcal{L}}\right)^j,$$
where the union is taken over all symplectic leaves  $\mathcal{L}\subset X_0$ with
$\dim \mathcal{L}<2j$. Here  $\I^{1,\mathcal{L}}_{\param^i}\subset \A^{\mathcal{L}}_{\param^i}$ (the slice
algebra corresponding to the leaf $\mathcal{L}$) is the ideal constructed similarly to
$\I^1_{\param^i}\subset \A_{\param^i}$.
Similarly to \cite[Lemma 5.2]{rouq_der}, we see that, for a Weil generic
$\widehat{\lambda}^i\subset \param^i$, the ideal $\I^j_{\widehat{\lambda}^i}$ is the
minimal ideal $\I\subset \A_{\widehat{\lambda}^i}$ with $\operatorname{GK-}\dim
\A_{\widehat{\lambda}^i}/\I<2j$.

Now, similarly to the proof of \cite[Theorem 6.1]{rouq_der}, Theorem
\ref{Thm:perv} follows from the next proposition. Here we pick a Zariski
generic $\lambda^1\in \param^1$ and set $\A_1:=\A_{\lambda^1},
\A_2:=\A_{\lambda^1+\chi}, \B:=\A^{(\theta^2)}_{\lambda^1,\chi},
\I_1^j:=\I^j_{\lambda^1}, \I_2^j:=\I^j_{\lambda^1+\chi}$.

\begin{Prop}\label{Prop:perv_techn}
The following is true.
\begin{itemize}
\item[(a)] For all $i,j$, we have $\I_2^j\operatorname{Tor}^{\A_1}_i(\B, \A_1/\I_1^j)=0$.
\item[(b)] For all $i,j$, we have $\operatorname{Tor}^{\A_2}_i(\A_2/\I_2^j, \B)\I_1^j=0$.
\item[(c)] We have   $\operatorname{Tor}^{\A_1}_i(\B, \A_1/\I_1^j)=0$
  for $i<n+1-j$.
\item[(d)] We have $\I_2^{j-1}\operatorname{Tor}^{\A_1}_i(\B, \A_1/\I^j_1)=
\operatorname{Tor}^{\A_2}_i(\A_2/\I_2^j,\B)\I_1^{j-1}=0$
for $i>n+1-j$.
\item[(e)]  Set $\B_{j}:=\operatorname{Tor}^{\A_1}_{n+1-j}(\B, \A_1/\I^j_1)$.
The kernel and the cokernel of the natural homomorphism $$\B_{j}\otimes_{\A_1}
\operatorname{Hom}_{\A_2}(\B_j, \A_2/\I_2^j)\rightarrow \A_2/\I_2^j$$
are annihilated by $\I_2^{j-1}$ on the left and on the right.
\item[(f)] The kernel and the cokernel of the natural homomorphism
$$\operatorname{Hom}_{\A_1}(\B_j, \A_1/\I_1^j)\otimes_{\A_2}\B_{j}
\rightarrow \A_1/\I_1^j.$$ 
are annihilated on the left and on the right by $\I_1^{j-1}$.
\end{itemize}
\end{Prop}
\begin{proof}
The proof of this proposition closely follows that of \cite[Proposition 6.3]{rouq_der}.
As in that proof (see Step 4 there), it is enough to prove (a)-(f) in the case when $\lambda^1$ is Weil
generic in $\param^1$. The proof that (a),(b) hold is the same as in Step 1 of the
proof of \cite[Proposition 6.3]{rouq_der}. To prove (c)-(f), we start with
the case of $j=1$. Here these claims follow from Corollary \ref{Cor:fin_dim_shift}. Now
the proof for arbitrary $j$ repeats that of Step 3 of the proof of \cite[Proposition 6.3]{rouq_der},
where we use  Proposition \ref{Prop:WC_restr} to show that the restriction of the wall-crossing
bimodule is still a wall-crossing bimodule (note that we can take sufficiently ample $\chi$
in the definition of a wall-crossing bimodule and hence the  cohomology vanishing
required in Proposition \ref{Prop:WC_restr} holds).
\end{proof}

\begin{Rem}\label{Rem:non_conical_slices}
All varieties $X^0$ we know have conical slices. One can prove Theorem \ref{Thm:perv} even without this assumption,
but the proof is considerably more technical.
\end{Rem}

\subsection{Application to Etingof's conjecture}
Here we consider a quiver $Q$ of affine type. We use the notation from
Section \ref{SS_quiver}.

We consider the category $\A^\theta_\lambda(v,w)\operatorname{-mod}_{\rho^{-1}(0)}$ of all coherent $\A^\theta_\lambda(v,w)$-modules supported at $\rho^{-1}(0)$.
We are going to describe $K_0(\A^\theta_\lambda(v,w)\operatorname{-mod}_{\rho^{-1}(0)})$ (we always consider complexified
$K_0$)  confirming \cite[Conjecture 1.1]{BL} when $Q$ is affine. The dimension of this $K_0$ coincides with the number of finite dimensional irreducible representations of $\A_\lambda(v,w)$ provided $\lambda$
the homological dimension of $\A_\lambda(v,w)$ is finite, see, e.g., \cite[Section 1.5]{BL}.

Let us write $\nu$ for the dominant weight of $\g(Q)$ with labels $w_i$. Further, we set
$\nu=\omega-\sum_{i\in Q_0}v_i\alpha_i$, where we write $\alpha_i$ for the simple root
of $Q$ corresponding to $i\in Q_0$.
Recall that, by \cite{Nakajima}, the homology group $H_{mid}(\M^\theta(v,w))$ (where ``mid'' stands for $\dim_\C \M^\theta(v,w)$) is identified with the weight space $L_\omega[\nu]$ of weight $\nu$ in  the irreducible integrable $\g(Q)$-module $L_\omega$ with highest weight $\omega$.
Further, by \cite{BarGin}, we have a natural inclusion $K_0(\A_\lambda(v,w)\operatorname{-mod}_{\rho^{-1}(0)})\hookrightarrow H_{mid}(\M^\theta(v,w))$ given by the characteristic cycle map $\CC_\lambda$.  We want to describe the image of $\CC_\lambda$.

Following \cite[Section 3]{BL}, we define a subalgebra $\a(=\a_\lambda)\subset\g(Q)$ and an $\a$-submodule $L_\omega^{\a}\subset L_\omega$. By definition, $\a$ is spanned by the Cartan subalgebra $\mathfrak{t}\subset \g(Q)$ and all root spaces $\g_\beta(Q)$ where $\beta=\sum_{i\in Q_0}b_i\alpha^i$ is a real root with $\sum_{i\in Q_0}b_i\lambda_i\in \Z$. For $L_\omega^{\a}$ we take the $\a$-submodule of $L_\omega$ generated by the extremal weight spaces (those where the weight is conjugate to the highest one under the action of the Weyl group).

\begin{Thm}\label{Thm:counting}
Let $Q$ be of affine type.  The image of $K_0(\A^\theta_\lambda(v,w)\operatorname{-mod}_{\rho^{-1}(0)})$ in $L_\omega[\nu]$
under $\CC_\lambda$ coincides with $L_\omega^\a\cap L_\omega[\nu]$.
\end{Thm}

\begin{proof}[Proof of Theorem \ref{Thm:counting}]
It was checked in \cite[Section 3.4]{BL} that $L_\omega^{\a}\cap L_\omega[\nu]$ is contained in
the image of $\CC_\lambda$. According to \cite[Section 6]{BL}, to show the equality one needs to check that
there are no notrivial {\it extremal} finite dimensional modules (defined in \cite[Section 6.4]{BL}).
This, in turn, follows if one proves
that the wall-crossing through the wall $\ker\delta$ (where $\delta$ stands for the indecomposable
imaginary root of $Q$) cannot have a homological shift of
$\frac{1}{2}\dim \M^\theta(v,w)$. This reduction was obtained in \cite[Section 8]{BL}.

So let us check that the homological shift for the wall-crossing to $\ker\delta$
is less than $\dim \M^\theta(v,w)/2$.
Thanks to Theorem \ref{Thm:perv}, it is enough to prove the following. Let $\param^1$ be an affine
subspace in $\tilde{\param}$ with associated vector space $\ker\delta$.
We need to show that the ideal $\I^1_{\param^1}$ coincides with the algebra $\A_{\param^1}(v,w)$
(at least after localization to a Zariski generic locus).
The quotient $\A_{\param^1}(v,w)/\I^1_{\param^1}$ is finitely generated over $\C[\param^1]$
and so is $\gr\left(\A_{\param^1}(v,w)/\I^1_{\param^1}\right)$. But $\gr \I^1_{\param^1}$
is a Poisson ideal. What remains to prove is that the variety $\M_p(v,w)$ has no symplectic leaves
that are single points as long as $p\in \ker\delta$ is Zariski generic (provided $\M_p(v,w)$ is not a
point itself). This follows from Lemma \ref{Lem:leaves}.
\end{proof}

\subsection{Wall-crossing bijections and annihilators}\label{SS_wc_bij}
We use the notation of Theorem \ref{Thm:perv}. Being perverse, the wall-crossing
functor $\WC_{\lambda^2\leftarrow \lambda^1}$ induces a bijection
$\mathfrak{wc}_{\lambda^2\leftarrow \lambda^1}:\operatorname{Irr}(\A_{\lambda^1})\rightarrow
\operatorname{Irr}(\A_{\lambda^2})$ between the sets of irreducible modules
(to be called the {\it wall-crossing bijection}). In this section, we are going to
investigate a compatibility of these bijections with the annihilators.

The following proposition generalizes the left cell part of  \cite[Theorem 1.1(i)]{cacti}.

\begin{Prop}\label{Prop:wc_bij_annih}
Let $N^1,N^2\subset \operatorname{Irr}(\A_{\lambda^1})$ be such that $\operatorname{Ann}_{\A_{\lambda^1}}(N^1)
= \operatorname{Ann}_{\A_{\lambda^1}}(N^2)$. Let $M^i:=\mathfrak{wc}_{\lambda^2\leftarrow \lambda^1}(N^i), i=1,2$.
Then $\operatorname{Ann}_{\A_{\lambda^2}}(M^1)= \operatorname{Ann}_{\A_{\lambda^2}}(M^2)$.
\end{Prop}
\begin{proof}
Let us write $\A_j$ for $\A_{\lambda^j}$ and $\B$ for the wall-crossing $\A_2$-$\A_1$-bimodule.
Let $\J:=\operatorname{Ann}_{\A_1}(N^i),  i=1,2$.

Note that $\B\otimes^L_{\A_1}\bullet$ is a perverse equivalence between
$\operatorname{HC}(\A_1)$ and $\operatorname{HC}(\A_2\text{-}\A_1)$
(viewed as hearts of standard t-structures on the full subcategories
of $D^b(\A_1\operatorname{-bimod}),D^b(\A_2\text{-}\A_1\operatorname{-bimod})$
of all complexes with HC homology). The filtrations are again defined by the
annihilation by the ideals $\I^j_{\lambda^1},\I^j_{\lambda^2}$ from the left.

Consider the HC $\A_1$-bimodule $\A_1/\J$. The ideal $\J$ is primitive hence prime
and so, by classical results of Borho and Kraft, \cite[Corollar 3.6]{BoKr}, the inclusion $\J\subsetneq \tilde{\J}$
implies $\operatorname{GK-}\dim \A_1/\J>\operatorname{GK-}\dim \A_1/\tilde{\J}$. It follows
that the HC bimodule $\A_1/\J$ has simple socle, say $S$. Let $T$ be the corresponding
simple $\A_2$-$\A_1$-bimodule. We claim that $\operatorname{Ann}_{\A_2}(M^j)$ coincides with
the left annihilator of $T$.

First of all, note that $\operatorname{Ann}_{\A_2}(M^j), \operatorname{LAnn}_{\A_2}(T)$
are primitive ideals (here we write $\operatorname{LAnn}$ for the left annihilator).
Moreover,  $$\VA(\A_2/\operatorname{Ann}_{\A_2}(M^j))=\VA(\A_2/\operatorname{LAnn}_{\A_2}(T))=\VA(\A_1/\J).$$
As in the proof of   \cite[Theorem 7.2(2)]{BL}, we see that $M^j$ is the head
of $\B_\ell\otimes_{\A_1}N^j$ and $T$ is the head of $\B_{\ell}\otimes_{\A_1}S$,
where $\B_\ell$ is a suitable Tor as in Proposition \ref{Prop:perv_techn}.
Note that $S\otimes_{\A_1} N^j\twoheadrightarrow N^j$. By axiom (P3) in the
definition of a perverse equivalence,  the kernels
of $\B_{\ell}\otimes_{\A_1}N^j\twoheadrightarrow M^j$ are annihilated by
$\I^2_{\ell+1}$ and the same is true for the kernel of $\B_{\ell}\otimes_{\A_1}S\twoheadrightarrow T$.
So we get epimorphisms $T\otimes_{\A_1}N^j\twoheadrightarrow M^j$. From here
we see that $\operatorname{LAnn}_{\A_2}(T)\subset \operatorname{Ann}_{\A_2}(M^j)$.
Since the associated varieties of these primitive ideals coincide, we apply
the result of Borho and Kraft again, and get $\operatorname{LAnn}_{\A_2}(T)=
\operatorname{Ann}_{\A_2}(M^j)$.
%
%
%
\end{proof}

\section{Wall-crossing functors as partial Ringel dualities}\label{S_Ringel}
\subsection{Highest weight categories}\label{SS_HW_cat}
Let us start by recalling the standard notion of a highest weight category.

{\it Basic assumptions}.  Let $\Cat$ be a $\K$-linear abelian category equivalent to
the category of finite  dimensional modules over a  unital associative finite dimensional $\K$-algebra.
Let $\mathcal{T}$ be an indexing set for the simples in $\Cat$, we write $L(\tau)$ for the simple object
indexed by $\tau$ and $P(\tau)$ for its projective cover.

By a highest weight structure on $\mathcal{C}$ we mean a pair $(\Cat,\leqslant)$, where $\leqslant$ is a partial order on $\mathcal{T}$ that satisfies the axioms (HW1) and (HW2) below. For $\tau\in
\mathcal{T}$, let $\Cat_{\leqslant \tau}$ (resp., $\Cat_{<\tau}$) denote the Serre span
of $L(\tau')$ with $\tau'\leqslant \tau$ (resp., $\tau'<\tau$). Here is our first axiom:
\begin{itemize}
\item[(HW1)] The quotient category $\Cat_{\leqslant \tau}/\Cat_{<\tau}$ is equivalent to the category
of vector spaces.
\end{itemize}
Let $\Delta(\tau)$ denote the projective cover of $L(\tau)$ in $\Cat_{\leqslant \tau}$.
Note that we have a natural epimorphism $P(\tau)\twoheadrightarrow \Delta(\tau)$. Here is our second
axiom:
\begin{itemize}
\item[(HW2)] The kernel of $P(\tau)\twoheadrightarrow \Delta(\tau)$ is filtered with  $\Delta(\tau')$, where $\tau'>\tau$.
\end{itemize}

Recall that in any highest weight category one has costandard objects $\nabla(\tau), \tau \in \mathcal{T},$
with $\dim \Ext^i(\Delta(\tau),\nabla(\tau'))=\delta_{i,0}\delta_{\tau,\tau'}$. By a {\it tilting}
in $\Cat$ we mean an object that is standardly filtered (admits a filtration by $\Delta$'s) and also
costandardly filtered. The indecomposable tilting objects are indexed by $\mathcal{T}$: we have a unique indecomposable
tilting $T(\tau)$ that admits an embedding $\Delta(\tau)\hookrightarrow T(\tau)$ with standardly filtered
cokernel.

Now let us recall the notion of Ringel duality that we will be generalizing below.
Let $\Cat_1,\Cat_2$ be two highest weight categories. Let $\Cat_2^\Delta, \Cat_1^\nabla$ denote the full
subcategories of standardly and costandardly filtered objects in $\Cat_2,\Cat_1$, respectively. Let
$R$ be an equivalence  $\Cat_1^\nabla\xrightarrow{\sim} \Cat_2^\Delta$ of exact categories.
Let $T$ denote the tilting generator
of $\Cat_1$, i.e., the sum of all indecomposable tilting objects. Then $\Cat_2$ gets identified
with $\End(T)^{opp}\operatorname{-mod}$ and the equivalence $R$ above becomes $\Hom(T,\bullet)$.
We also have a derived equivalence $R\Hom(T,\bullet): D^b(\Cat_1)\rightarrow D^b(\Cat_2)$. This equivalence
maps injectives to tiltings and, obviously, tiltings to projectives.  We write $\Cat_1^\vee$ for $\Cat_2$.
The functor $R$ is called the (covariant) {\it Ringel duality}, and the category $\Cat_1^\vee$
is called the {\it Ringel dual} of $\Cat_1$.
%
%

\subsection{Categories $\mathcal{O}$ for symplectic resolutions and cross-walling functors}\label{SS_cat_O_CW}
Now let us recall an example of a highest weight category.

Suppose that we have a conical symplectic resolution $X$ that comes equipped with a Hamiltonian action of a torus $T$
that commutes with the contracting $\C^\times$-action. Let $\lambda\in \tilde{\param}$. The action of
$T$ on $\mathcal{O}_X$ lifts to a Hamiltonian action of $T$ on $\A_\lambda^\theta$.
So we get a Hamiltonian action on $\A_\lambda$.
By $\Phi$ we denote the quantum comoment map $\mathfrak{t}\rightarrow \A_\lambda$.

Let $\nu:\C^\times\rightarrow T$ be a one-parameter subgroup. The subgroup $\nu$ induces
a grading $\A_\lambda=\bigoplus_{i\in \Z}\A_\lambda^{i,\nu}$. We set $\A_\lambda^{\geqslant 0,\nu}=\bigoplus_{i\geqslant 0}\A_\lambda^{i,\nu}$ and define $\A_\lambda^{>0,\nu}$ similarly. Further, set $\Ca_{\nu}(\A_\lambda):=\A^{0,\nu}_\lambda/\bigoplus_{i>0}\A^{-i,\nu}_\lambda \A^{i,\nu}_\lambda$.
Note that $\A_\lambda/\A_\lambda\A_{\lambda}^{>0,\nu}$ is an $\A_\lambda$-$\Ca_{\nu}(\A_\lambda)$-bimodule,
while $\A_\lambda/\A_{\lambda}^{<0,\nu}\A_\lambda$ is a $\Ca_{\nu}(\A_\lambda)$-$\A_\lambda$-bimodule.

Define the category $\OCat_{\nu}(\A_\lambda)$ as the full subcategory of $\A_\lambda\operatorname{-mod}$
consisting of all modules, where the action of $\A_{\lambda}^{>0,\nu}$ is locally nilpotent.
We get two functors $\Delta_{\nu},\nabla_{\nu}:\Ca_{\nu}(\A_\lambda)\operatorname{-mod}
\rightarrow \OCat_{\nu}(\A_\lambda)$ given by $$\Delta_{\nu}(N):=(\A_\lambda/\A_{\lambda}\A_{\lambda}^{>0,\nu})\otimes_{\Ca_{\nu}(\A_\lambda)}N,
\nabla_{\nu}(N):=\Hom_{\Ca_{\nu}(\A_\lambda)}(\A_\lambda/\A_{\lambda}^{<0,\nu}\A_\lambda, N).$$

Now suppose that $T$ acts on $X$ with finitely many fixed points. We say that a one-parameter
group $\nu:\C^\times\rightarrow T$ is {\it generic} if $X^{\nu(\C^\times)}=X^T$.
Equivalently, $\nu$ is generic if and only if it does not lie in $\ker\kappa$ for any character $\kappa$ of the
$T$-action on $\bigoplus_{p\in X^T}T_pX$. The hyperplanes $\ker\kappa$ split the
lattice $\Hom(\C^\times,T)$ into the union of polyhedral regions to be called
{\it chambers} (of one-parameter subgroups).

Suppose that $\nu$ is generic. Further, pick a generic, see Definition \ref{defi:chamb_termin},
$\theta\in \tilde{\param}_{\Q}$ and $\lambda_0\in \tilde{\param}$. Let $\lambda:=\lambda_0+n\theta$
for $n\gg 0$.  

\begin{Prop}\label{Prop:cat_O}
The following is true:
\begin{enumerate}
\item The category $\OCat_\nu(\A_\lambda)$ only depends on the chamber of $\nu$.
\item The natural functor $D^b(\OCat_\nu(\A_\lambda))\rightarrow D^b(\A_\lambda\operatorname{-mod})$
is a full embedding.
\item $\Ca_{\nu}(\A_\lambda)=\C[X^T]$.
\item More generally, we have $\Ca_{\nu_0}(\A_\lambda)=\bigoplus_{Z} \A^Z_{\iota^*_Z(\lambda)-\rho_Z}$,
where the summation is taken over the irreducible components $Z$ of $X^{\nu_0(\C^\times)}$,
$\iota_Z$ is the embedding $Z\hookrightarrow X$, $\iota_Z^*:H^2(X,\C)\rightarrow H^2(Z,\C)$
is the corresponding pull-back map, $\rho_Z$ is a suitable element of $H^2(Z,\C)$
and $\A_{\iota^*_Z(\lambda)-\rho_Z}^Z$ stands for the global sections of the filtered quantization of $Z$ with period
$\iota^*_Z(\lambda)-\rho_Z$.
\item The category $\OCat_\nu(\A_\lambda)$ is highest weight, where the standard objects are
$\Delta_\nu(p)$, the costandard objects are $\nabla_{\nu}(p)$, where $p\in X^T$.
For an order, which is a part of the definition of a highest weight structure, we take
the contraction order on $X^T$ given by $\nu$.
\item Suppose $\nu_0$ lies in the face of a chamber containing $\nu$. Then $\Delta_{\nu_0},
\nabla_{\nu_0}$ restrict to exact functors $\OCat_\nu(\Ca_{\nu_0}(\A_\lambda))\rightarrow
\OCat_\nu(\A_\lambda)$.
\item The functor $\WC^{-1}_{\lambda\leftarrow \lambda^-}:D^b(\OCat_\nu(\A_\lambda))\rightarrow
D^b(\OCat_{\nu}(\A_{\lambda^-}))$ is a Ringel duality functor.
\end{enumerate}
\end{Prop}
\begin{proof}
(1) follows from \cite[Corollary 3.19]{BLPW}. (2) is \cite[Corollary 5.13]{BLPW}.
(3) is \cite[Proposition 5.3]{BLPW}. (4) follows from \cite[Propositions 5.3,5.7]{CWR}.
(5) follows from \cite[Theorem 5.12]{BLPW}. (6) follows from \cite[Proposition 6.9, Section 6.5]{CWR}.
(6) is \cite[Proposition 7.7]{CWR}.
\end{proof}

Let us recall the cross-walling (a.k.a. shuffling) functors introduced in \cite[Section 8]{BLPW}
and studied in more detail in \cite{CWR}. Let $\nu,\nu'$ be two generic one-parameter subgroups.
Then there is a unique functor $\CW_{\nu'\leftarrow \nu}:D^b(\OCat_\nu(\A_\lambda))\rightarrow
D^b(\OCat_{\nu'}(\A_\lambda))$ with the property that
$$\Hom_{D^b(\A_\lambda\operatorname{-mod})}(M,N)=\Hom_{D^b(\OCat_{\nu'}(\A_\lambda))}(\CW_{\nu'\leftarrow \nu}M,N).$$
This was proved in \cite[Section 8.2]{BLPW}

The following results were obtained in \cite[Section 7]{CWR}. We choose a parameter $\lambda$
in the same way as for Proposition \ref{Prop:cat_O}.

\begin{Prop}\label{Prop:CW_properties}
The functor $\CW_{\nu'\leftarrow \nu}$ has the following properties.
\begin{enumerate}
\item The functor is an equivalence for all $\nu,\nu'$.
\item Suppose that a sequence $\nu,\nu',\nu''$ is reduced (meaning that any wall that does not separate
$\nu,\nu''$ does not separate $\nu,\nu'$ either). Then $\CW_{\nu''\leftarrow \nu}\xrightarrow{\sim}
\CW_{\nu''\leftarrow \nu'}\circ \CW_{\nu'\leftarrow \nu}$.
\item The functor $\CW_{-\nu\leftarrow \nu}[\dim X/2]:D^b(\OCat_\nu(\A_\lambda))\xrightarrow{\sim}
D^b(\OCat_{-\nu}(\A_\lambda))$ is a  Ringel duality functor.
\item Let $\nu_0$ be a one-parameter subgroup lying in the faces of the chambers
of $\nu,\nu'$. Then the functors $\CW_{\nu'\leftarrow \nu}\circ \Delta_{\nu_0}$
and $\Delta_{\nu_0}\circ \underline{\CW}_{\nu'\leftarrow \nu}$ from
$D^b(\OCat_\nu(\Ca_{\nu_0}(\A_\lambda)))$ to $D^b(\OCat_{\nu'}(\A_\lambda))$
are naturally isomorphic. Here $\underline{\CW}_{\nu'\leftarrow \nu}$ stands for
the cross-walling functor $D^b(\OCat_\nu(\Ca_{\nu_0}(\A_\lambda)))\xrightarrow{\sim}
D^b(\OCat_{\nu'}(\Ca_{\nu_0}(\A_\lambda)))$.
\end{enumerate}
\end{Prop}

\subsection{Standardly stratified categories}
Here we are going to a introduce standardly stratified structures
generalizing highest weight ones. We follow \cite{LW}. The definition given there
is more restrictive than in \cite{CPS} but is less restrictive than in \cite{ADL}.

Let $\Cat, \mathcal{T},L(\tau),P(\tau)$ have the same meaning as in Basic assumptions of Section
\ref{SS_HW_cat}.  The additional structure of a standardly stratified category on $\mathcal{C}$ is a
partial  {\it pre-order} $\leqslant$ on $\mathcal{T}$ that should satisfy certain axioms to be explained below.
Let us write $\Xi$ for the set of equivalence classes of $\leqslant$,
this is a poset (with partial order again denoted by $\leqslant$) that comes with a natural surjection $\varrho:\mathcal{T}\twoheadrightarrow \Xi$. The pre-order $\leqslant$ defines a filtration on $\Cat$ by Serre subcategories indexed by $\Xi$. Namely, to $\xi\in \Xi$ we assign the subcategories $\Cat_{\leqslant \xi}$ that is the Serre span
of the simples $L(\tau)$ with $\varrho(\tau)\leqslant \xi$. Define $\Cat_{<\xi}$ analogously and let
$\Cat_\xi$ denote the quotient $\Cat_{\leqslant \xi}/\Cat_\xi$. Let $\pi_\xi$ denote the quotient
functor $\Cat_{\leqslant \xi}\twoheadrightarrow \Cat_{\xi}$. Let us write $\Delta_\xi:\Cat_\xi\rightarrow \Cat_{\leqslant \xi}$ for the left adjoint functor of $\pi_\xi$. Also we write $\gr\Cat$ for $\bigoplus_{\xi}\Cat_\xi, \Delta$
for $\bigoplus_\xi \Delta_\xi:\gr\Cat\rightarrow \Cat$. We call $\Delta$ the
{\it standardization functor}. Finally, for $\tau\in \varrho^{-1}(\xi)$
we write $L_\xi(\tau)$ for $\pi_\xi(L(\tau))$, $P_\xi(\tau)$ for the projective cover of
$L_\xi(\tau)$ in $\Cat_\xi$ and $\Delta(\tau)$ for $\Delta_\xi(P_\xi(\tau))$.
The object $\Delta(\tau)$ is called {\it standard}. The object $\bDelta(\tau):=\Delta_\xi(L_\xi(\tau))$
is called {\it proper standard}. Note that there is a natural epimorphism $P(\tau)\twoheadrightarrow
\Delta(\tau)$.

The axioms to be satisfied by $(\Cat,\leqslant)$ in order to give a standardly stratified structure are as follows.
\begin{itemize}
\item[(SS1)] The functor $\Delta:\gr\Cat\rightarrow \Cat$ is exact.
\item[(SS2)] The projective $P(\tau)$ admits an epimorphism onto $\Delta(\tau)$ whose kernel
has a filtration with successive quotients $\Delta(\tau')$, where $\tau'>\tau$.
\end{itemize}

We will also need the notion of a {\it weakly} standardly stratified category. Here we keep (SS1) but use a
weaker version of (SS2):
\begin{itemize}
\item[(SS2$'$)] The projective $P(\tau)$ admits an epimorphism onto $\Delta(\tau)$ whose kernel
admits a filtration with successive quotients $\Delta_\xi(M_\xi)$, where $\xi>\varrho(\tau)$
and $M_\xi$ is some object in $\Cat_\xi$.
\end{itemize}

Note that (SS1) allows to identify $K_0(\gr\Cat)$ and $K_0(\Cat)$ by means of $\Delta$. If (SS2)
also holds, then we also have the identification of $K_0(\gr\Cat\operatorname{-proj})$
and $K_0(\Cat\operatorname{-proj})$.

If all quotient categories $\Cat_\xi$ are equivalent to $\operatorname{Vect}$, then
a (weakly) standardly stratified category is the same as a highest weight category. On
the opposite end, if we take the trivial pre-order on $\mathcal{T}$, then there is no
additional structure.

\subsubsection{(Proper) standardly filtered objects}
We say that an object in $\Cat$ is standardly filtered if it admits a filtration whose successive quotients are
standard. The notion of a proper standardly filtered object is introduced in a similar fashion. The categories
of the standardly filtered and of the proper standardly filtered objects will be denoted by $\Cat^{\Delta}$
and $\Cat^{\bDelta}$. Note that (SS1) implies that $\Cat^{\Delta}\subset \Cat^{\bDelta}$.

\begin{Lem}\label{Lem:prop_stand_stratif}
Suppose (SS1) holds.  Let $M$ be an object in $\Cat^{\bDelta}$ such that all proper standard quotients are of the form $\bDelta(\tau)$ with $\varrho(\tau)=\xi$. Then $M=\Delta_\xi(\pi_\xi(M))$.
\end{Lem}

This is \cite[Lemma 3.1]{CWR}.


Also note that in a weakly standardly stratified category the following hold:
\begin{align}\label{Ext:vanish}
&\Ext^i_{\Cat}(\Delta_\xi(M),\Delta_{\xi'}(N))\neq 0\Rightarrow \xi'\leqslant \xi.\\\label{Ext:equal}
&\Ext^i_{\Cat}(\Delta_\xi(M),\Delta_{\xi}(N))=\Ext^i_{\Cat_\xi}(M,N).
\end{align}

\subsubsection{Subcategories and quotients}
Suppose that $\Cat$ is weakly standardly stratified.

Let $\Xi_0$ be a poset ideal in $\Xi$. Let $\Cat_{\Xi_0}$ denote the Serre span
of the simples $L(\tau)$ with $\varrho(\tau)\in \Xi_0$. Then $\Cat_{\Xi_0}$
is a standardly stratified category with pre-order on $\mathcal{T}_0:=\varrho^{-1}(\Xi_0)$
restricted from $\Xi$. Note that, for $\tau\in \mathcal{T}_0$, we have
$\bDelta_{\Xi_0}(\tau)=\bDelta(\tau),
\Delta_{\Xi_0}(\tau)=\Delta(\tau)$,
where the subscript $\Xi_0$ refers to the objects computed in $\Cat_{\Xi_0}$.

The embedding $\iota_{\Xi_0}:\Cat_{\Xi_0}^{\bDelta}\hookrightarrow \Cat^{\bDelta}$
admits a left adjoint functor $\iota^!_{\Xi_0}$, to an object $M\in \Cat^{\bDelta}$,
this functor assigns the maximal quotient lying in $\Cat_{\Xi_0}^{\bDelta}$.

Now let $\Cat^{\Xi_0}$ be the quotient category $\Cat/\Cat_{\Xi_0}$. Let $\pi_{\Xi_0}$
denote the quotient functor $\Cat\rightarrow \Cat^{\Xi_0}$ and let $\pi_{\Xi_0}^!$
be its left adjoint. The category $\Cat^{\Xi_0}$ is standardly stratified with
pre-order on $\Lambda^0:=\Lambda\setminus \Lambda_0$ restricted from $\Lambda$.
For $\xi\in \Xi^0:=\Xi\setminus \Xi_0$ we have $\Delta^0_\xi=\pi_{\Xi_0}\circ \Delta_\xi$.
Let us also point out that $\pi_{\Xi_0}^!$ defines a full embedding $(\Cat^{\Xi_0})^{\bDelta}
\hookrightarrow \Cat^{\bDelta}$ whose image coincides with the full subcategory
$\Cat^{\bDelta,\Lambda^0}$ consisting of all objects that admit a filtration with successive
quotients $\bDelta(\tau)$ with $\tau\in \mathcal{T}^0$. This embedding sends $\Delta^{\Xi_0}(\tau)$
to $\Delta(\tau)$, $P^{\Xi_0}(\tau)$ to $P(\tau)$.

The following lemma describes the derived categories of $\Cat^{\Xi_0}$ and $\Cat_{\Xi_0}$.

\begin{Lem}\label{Lem:der_cat_stand}
A natural functor $D^b(\Cat_{\Xi_0})\hookrightarrow D^b(\Cat)$ is a fully faithful embedding.
Moreover, $D^b(\Cat^{\Xi_0})=D^b(\Cat)/D^b(\Cat_{\Xi_0})$.
\end{Lem}

See \cite[Lemma 3.2]{CWR} for a proof.

\subsubsection{Opposite category}
Let $\Cat$ be a standardly stratified category.
It turns out that the opposite category $\Cat^{opp}$ is also standardly stratified with the same pre-order
$\leqslant$, see \cite[1.2]{LW}. The standard and proper standard objects for $\Cat^{opp}$ are denoted
by $\nabla(\tau)$ and $\bnabla(\tau)$, when viewed as objects of $\Cat$, they are called
{\it costandard} and {\it proper costandard}. The right adjoint functor to $\pi_\xi$ will be
denoted by $\nabla_\xi$ and we write $\nabla$ for $\bigoplus_\xi \nabla_\xi$ (this is the so called
{\it costandardization functor}). So we have $\bnabla(\tau)=\nabla(L_\xi(\tau))$ and
$\nabla(\tau)=\nabla(I_\xi(\tau))$, where $I_\xi(\tau)$ is the injective envelope of $L_\xi(\tau)$ in $\Cat_\xi$.

Let us write $\Cat^\nabla, \Cat^{\bnabla}$ for the subcategories of costandardly and of proper costandardly
filtered objects. We have the following standard lemma (that was used in \cite{LW} to verify the claims
in the previous paragraph).

\begin{Lem}[Lemma 2.4 in \cite{LW}]\label{Lem:exts}
The following is true.
\begin{enumerate}
\item $\dim \Ext^i(\Delta(\tau), \bnabla(\tau'))=\dim
  \Ext^i(\bDelta(\tau), \nabla(\tau'))=
  \delta_{i,0}\delta_{\tau,\tau'}$. 
 \item For $N\in \Cat$, we have $N\in \Cat^\nabla$ (resp., $N\in
  \Cat^{\bnabla}$) if and only if $\Ext^1(\bDelta(\tau),N)=0$
  (resp., $\Ext^1(\Delta(\tau), N)=0$) for all $\tau$. Similar characterizations
  are true for $\Cat^{\Delta},\Cat^{\bDelta}$.
\end{enumerate}
\end{Lem}

Let us also note the following fact.

\begin{Lem}\label{Lem:w_st_stratif}
Let $\Cat$ be a weakly standardly stratified category. Then $\Cat$ is standardly stratified
if and only if the right adjoint $\nabla_\xi$ of $\pi_\xi$ is exact.
\end{Lem}

This is  \cite[Lemma 3.4]{CWR}.

\subsubsection{Equivalences}
Let $\Cat_1,\Cat_2$ be two weakly standardly stratified categories and let $\Xi_1,\Xi_2$ be the corresponding
posets. By an equivalence of $(\Cat_1,\Xi_1),(\Cat_2,\Xi_2)$ we mean a pair $(\Phi,\phi)$ consisting
of an equivalence $\Phi:\Cat_1\rightarrow \Cat_2$ of abelian categories and a poset isomorphism
$\phi: \Xi_1\rightarrow \Xi_2$ such that the bijection between the simples induced by $\Phi$
is compatible with $\phi$. Clearly, $\Phi$ induces an equivalence $\gr\Phi:\gr\Cat_1\xrightarrow{\sim}\gr\Cat_2$
and $\Phi\circ \Delta\cong \Delta\circ \gr\Phi, \Phi\circ\nabla\cong \nabla\circ \gr\Phi$.

\subsubsection{Example: Categorical tensor products}\label{SSS_cat_tens_prod}
The formalism of standardly stratified categories was introduced in \cite{LW} to treat
tensor products of categorical representations of Kac-Moody algebras. Namely,
let $\g$ be a Kac-Moody algebra and $\mathcal{V}_1,\ldots,\mathcal{V}_k$
be minimal  $\g$-categorifications (in the sense of Rouquier, \cite{Rouquier_2Kac}) that categorify irreducible integrable
highest weight representations. Webster in \cite{Webster_tensor} constructed
the categorical tensor product $\mathcal{V}_1\otimes\ldots\otimes \mathcal{V}_k$
that was equipped in \cite{LW} with a structure of a standardly stratified
category. The reader is referred to \cite{LW} for details.

\subsubsection{Example: Standardly stratified structure on $\mathcal{O}$ from degenerating $\nu$}\label{SSS_SSC_1param}
Let us give another example of a standardly stratified category. Let
$X,T$ be as in  Section \ref{SS_cat_O_CW}. Pick a generic
$\nu:\C^\times\rightarrow T$ and let $\nu_0$ lie in the closure of the chamber containing $\nu$.
Then $\nu_0$ defines the order on the set of irreducible components
of $X^{\nu_0(\C^\times)}$ (by contraction, see \cite[Section 6.1]{CWR} for details).
So we get a pre-order $\leqslant_{\nu_0}$ on the set $X^T$.
It is easy to see (and was checked in \cite[Section 6.1]{CWR}) that the order
$\leqslant_{\nu}$ refines $\leqslant_{\nu_0}$.

Now pick a sufficiently dominant quantization parameter $\lambda$ and consider
the category $\OCat_\nu(\A_\lambda)$.

The following proposition is the main result of \cite[Section 6]{CWR}.

\begin{Prop}\label{Prop:stand_filt}
The pre-order $\leqslant_{\nu_0}$ defines a standardly stratified structure on
$\OCat_\nu(\A_\lambda)$. The associated graded category is $\OCat_\nu(\Ca_{\nu_0}(\A_\lambda))$.
The standardization functor is $\Delta_{\nu_0}$, while the costandardization
functor is $\nabla_{\nu_0}$.
\end{Prop}

\subsection{Standardly stratified structure on $\mathcal{O}$ from parameter deformation}
This is one of two central parts of this section. Here we introduce a new standardly
stratified structure on $\OCat_\nu(\A_\lambda)$ coming from a deformation of $\lambda$
along an affine subspace parallel to a face in the chamber containing $\lambda$.

\subsubsection{Main result}\label{SSS_main_result_stand_stratif}
Pick a face $\Gamma$ of  a classical chamber $C$ and an element $\lambda^\circ\in \tilde{\param}$.
Let $\param_0$ denote the vector subspace of $\tilde{\param}$ spanned by $\Gamma$. Set
$\param^1:=\lambda^\circ+\param_0$.

\begin{Lem}\label{Lem:hw_as_gen}
There is an asymptotically generic Zariski open subset $\widehat{\param}^1\subset \param^1$
with the following properties.
\begin{enumerate}
\item We have an algebra isomorphism $\Ca_\nu(\A_\lambda)\cong \C[X^T]$ for any $\lambda\in \widehat{\param}^1$.
\item For any $\lambda\in \widehat{\param}^1$, the category $\OCat_\nu(\A_\lambda)$ is highest weight with standard objects $\Delta_{\nu,\lambda}(p)$ and costandard objects $\nabla_{\nu,\lambda}(p)$.
\end{enumerate}
\end{Lem}
\begin{proof}
The existence of $\widehat{\param}$ such that (1) holds
follows from the proof of \cite[Proposition 5.3]{CWR}.

Now assume that (1) holds.
The category $\OCat_\nu(\A_\lambda)$ is highest weight with standards $\Delta_\nu(p)$
and costandards $\nabla_\nu(p)$ if and only if $\Ext^2(\Delta_{\nu,\lambda}(p),\nabla_{\nu,\lambda}(p'))=0$,
see the proof of \cite[Theorem 5.12]{BLPW}.  That $\Ext^2(\Delta_{\nu,\lambda}(p),\nabla_{\nu,\lambda}(p'))=0$  for
$\lambda$ in an asymptotically generic open subset follows from the proof in \cite[Appendix]{BLPW}.
\end{proof}

It follows that, possibly after replacing $\lambda^\circ$, with $\lambda^\circ+\chi$ for
$\chi\in \tilde{\param}_{\Z}\cap C$ we may assume that $\OCat_\nu(\A_{\lambda})$
is highest weight for a Zariski generic $\lambda\in \param^1$. Recall, \cite[Lemma 6.4]{CWR},
that the order is introduced as follows. Let $c_\lambda(p)$ denote the image of $h\in \A_\lambda^T$ under the maps
$\A_\lambda^T\twoheadrightarrow \C_\nu(\A_\lambda)\twoheadrightarrow \C_p$. We set
$p<_{\lambda}p'$ if $c_\lambda(p')-c_{\lambda}(p)\in \Z_{>0}$. Recall that $c_\lambda(p)-c_\lambda(p')$
is a linear function whose value at $\chi\in \tilde{\param}_{\Z}$ coincides with
$\alpha_p(\chi)-\alpha_{p'}(\chi)$, where we write $\alpha_p(\chi)$ for the character
of the action of $\nu$ in the fiber of $\mathcal{O}(\chi)$, see \cite[Lemma 6.4]{CWR}.

Now pick a sufficiently general integral point $\chi$ in the interior of $\Gamma$. For $\lambda=\lambda^\circ+ N\chi$
for $N\gg 0$, the order $<_\lambda$ can be described in the following way. Set $p\prec_\chi p'$
if $\alpha_{p'}(\chi)-\alpha_p(\chi)>0$ and $p\sim_\chi p'$ if $\alpha_p(\chi)=\alpha_{p'}(\chi)$.
For $p\sim_\chi p'$, the difference $c_{p'}(\lambda)-c_{p}(\lambda)$ is independent of $\lambda\in
\param^1$. In particular, if $p\sim_\chi p'$, then
we have $p<_\lambda p'$ if and only if $p<_{\widehat{\lambda}}p'$ for a Weil
generic $\widehat{\lambda}\in \param^1$. So we see that the order $<_\lambda$ is refined by the
following order $\leqslant$: we have $p<p'$ if $p<_\chi p'$ or $p<_{\widehat{\lambda}}p'$ (note that the latter
automatically implies $p\sim_\chi p'$). We also would like to point out that $\leqslant_\chi$ is
a pre-order on $X^T$ (refined by $\leqslant_\lambda$).

Here is the main result of the present section.

\begin{Prop}\label{Prop:stand_stratif}
Let $\lambda$ be as above. The category $\OCat_\nu(\A_\lambda)$ is standardly stratified with
respect to the pre-order $\leqslant_\chi$. We have a labeling preserving equivalence
$\gr\OCat_\nu(\A_\lambda)\cong \OCat_\nu(\A_{\widehat{\lambda}})$ for a Weil generic
$\widehat{\lambda}\in \param^1$, in particular, the right hand side is independent of
$\widehat{\lambda}$.
\end{Prop}

We prove this result in the rest of the section.

\subsubsection{Objects $\bDelta_{\param^1}(p),
\bnabla_{\param^1}(p)$}
Note that every $p\in X^T$ defines an algebra homomorphism $\Ca_{\nu}(\A_{\param^1})\rightarrow \C[\param^1]$
that is surjective over a non-empty open subset in $\param^1$. This allows to define the
$\A_{\param^1}$-modules $\Delta_{\param^1}(p), \nabla_{\param^1}(p)$ that specialize
to $\Delta_\lambda(p), \nabla_{\lambda}(p)$ for a Zariski generic $\lambda\in \param^1$.

We start by introducing $\A_{\param^1}$-modules $\bDelta_{\param^1}(p),\bnabla_{\param^1}(p)$
that have the following properties:
\begin{enumerate}
\item they are generically free over $\param^1$,
\item specialize to $ L_{\widehat{\lambda}}(p)$ at $\widehat{\lambda}$,
\item we have $\bDelta_\lambda(p)\twoheadrightarrow L_\lambda(p)$
(resp., $L_\lambda(p)\hookrightarrow \bnabla_\lambda(p)$) with kernel
(resp., cokernel) filtered with $L_\lambda(p')$ for $p'<_\chi p$.
\end{enumerate}
Later we will see that these objects specialized to $\lambda$ become  proper standard and proper costandard
for the standardly stratified structure on  $\OCat_\nu(\A_\lambda)$ we are going to produce.

The objects $\bDelta_{\param^1}(p)$ are constructed as quotients of $\Delta_{\param^1}(p)$
as in the proof of \cite[Lemma 3.3]{Cher_supp}. Namely, let us consider all labels
$p'\sim_\chi p, p'<_{\widehat{\lambda}}p$. Then we consider
the object $M_{\param^1}:=\bigoplus_{p'}\Delta_{\param^1}(p')$.
Then we define the object $\Delta^i_{\param^1}(p)$ recursively as follows:
$\Delta^0_{\param^1}(p)=\Delta_{\param^1}(p)$ and $\Delta^i_{\param^1}(p)$
is the cokernel of the natural homomorphism
$$\Hom_{\A_{\param^1}}(M_{\param^1},\Delta^{i-1}_{\param^1}(p))\otimes_{\C[\param_1]}M_{\param^1}\rightarrow
\Delta^{i-1}_{\param_1}(p).$$
For $i$ sufficiently large, the kernel of $\Delta^j_{\param^1}(p)\twoheadrightarrow
\Delta^{j+1}_{\param^1}(p)$ is $\C[\param^1]$-torsion for all $j\geqslant i$.
We take $\Delta^i_{\param^1}(p)$ for $\bDelta_{\param^1}(p)$.

The objects $\bDelta_{\param^1}(p)$ have  properties (1)-(3) similarly to \cite[Lemma 3.3]{Cher_supp}.

The modules $\bnabla_{\param^1}(p)$ are produced as follows.
We set $N_{\param^1}:=\bigoplus_{p'}\nabla_{\param^1}(p')$, $\nabla^0_{\param^1}(p):=\nabla_{\param^1}(p)$.
Take generators $\varphi_1,\ldots,\varphi_\ell$ of the $\C[\param^1]$-module
$$\Hom_{\A_{\param^1}}(\nabla^{i-1}_{\param^1}(p), N_{\param^1}).$$
Then for $\nabla^i_{\param^1}(p)$ we take the intersection of the kernels of
$\varphi_i, i=1,\ldots,\ell$. Define $\bnabla_{\param^1}(p)$ similarly to
the above.  As in \cite[Lemma 3.3]{Cher_supp}, one shows that the modules $\bnabla_{\param^1}(p)$
have  properties (1)-(3).

\subsubsection{Filtrations}
Now let us establish an important property of the objects $\Delta_{\param^1}(p),
\nabla_{\param^1}(p)$.

\begin{Lem}\label{Lem:Delta_filtr}
We have a principal Zariski open subset $\param^0\subset \param^1$
and a filtration $\Delta_{\param^0}(p)(:=\C[\param^0]\otimes_{\C[\param^1]}\Delta_{\param^1}(p))=
F_0\supsetneq F_1\supsetneq F_2\ldots\supsetneq F_k
\supsetneq F_{k+1}=\{0\}$ with the following properties:
\begin{itemize}
\item[(i)] $F_i/F_{i+1}$ is  free over $\param^0$.
\item[(ii)] $(F_0/F_1)_{\lambda^1}\cong \bDelta_{\lambda^1}(p)$
for  $\lambda^1\in \param^0$.
\item[(iii)] $(F_i/F_{i+1})_{\lambda^1}\cong \bDelta_{\lambda^1}(p_i)^{\oplus m_i}$
for  $\lambda^1\in \param^0$ and $p_i\sim_\chi p, p_i<_{\widehat{\lambda}}p$
and $m_i\in \Z_{>0}$ independent of $\lambda^1$.
\end{itemize}
\end{Lem}
\begin{proof}
The construction of the filtration is based on the construction of $\bDelta_{\param^1}(\lambda)$.
Namely, let us order the labels satisfying $p'\sim_\chi p, p'<_{\widehat{\lambda}}p$
in a non-increasing way: $p_s,p_{s-1},\ldots, p_0=p$ ($p_s$ is the smallest).
Then define $F_k$ as the image of $\Hom_{\A_{\param^1}}(\Delta_{\param^1}(p_s),\Delta_{\param^1}(p))\otimes
_{\param^1}\Delta_{\param^1}(p_s)\rightarrow \Delta_{\param^1}(p)$. It satisfies (i) and (iii)
(with $p_i=p_s$). Shrinking $\param^1$ to a principal Zariski open subset $\param^0$, we see that
$\Hom_{\A_{\param^0}}(\Delta_{\param^0}(p_s), \Delta_{\param^0}(p)/F_k)=0$. In particular, any
homomorphism $\Delta_{\param^0}(p_{s-1})\rightarrow \Delta_{\param^0}(p)/F_k$
factors through $\bDelta_{\param^0}(p_{s-1})$. We define $F_{k-1}/F_k$
as the image of $\Hom_{\A_{\param^1}}(\bDelta_{\param^1}(p_{s-1}),\Delta_{\param^1}(p))\otimes
_{\param^1}\bDelta_{\param^1}(p_{s-1})\rightarrow \Delta_{\param^1}(p)$.
Then we shrink $\param^0$. We continue considering homomorphisms from
$\Delta_{\param^0}(p_s),\bDelta_{\param^0}(p_{s-1})$ until we arrive
at the situation when both $\Hom_{\A_{\param^1}}(\bDelta_{\param^0}(p_s),\Delta_{\param^1}(p)/F_i),
\Hom_{\A_{\param^0}}(\bDelta_{\param^0}(p_{s-1}),\Delta_{\param^0}(p)/F_i)$ are zero
(we do arrive at this situation because the length of a Weil generic specialization of
$\Delta_{\param^0}(p)/F_i$ reduces after each step). In particular, any homomorphism
$\Delta_{\param^0}(p_{s-2})\rightarrow \Delta_{\param^0}(p)/F_i$ factors
through $\bDelta_{\param^0}(p_{s-2})$. Then we repeat the argument.
\end{proof}

The dual  statement holds for $\nabla_{\param^1}(p)$ (we consider the filtration
by $\bnabla_{\param^0}(p_i)$'s).

\subsubsection{Objects $\tilde{\Delta}_{\param^1}(p), \tilde{\nabla}_{\param^1}(p)$}
Now let us produce the objects  $\tilde{\Delta}_{\param^1}(p)$ that later will be shown
to be standard for the standardly stratified structure. Namely, let us order the labels
$p'$ with $p'\sim_\chi p$ in a non-decreasing way: $p_1>p_2>\ldots>p_n$.
Let us define the objects $\tilde{\Delta}_{\param^1}(p_i)_k, k\leqslant i,$ inductively.
We set $\tilde{\Delta}_{\param^1}(p_i)_i=\Delta_{\param^1}(p_i)$. If $\tilde{\Delta}_{\param^1}(p_i)_k$
is already defined, then for $\tilde{\Delta}_{\param^1}(p_i)_{k-1}$ we take
the universal extension
$$0\rightarrow \Ext^1_{\A_{\param_1}}(\tilde{\Delta}_{\param^1}(p_i)_k, \Delta_{\param^1}(p_{k-1}))\otimes_{\C[\param^1]}\Delta_{\param^1}(p_{k-1})
\rightarrow \tilde{\Delta}_{\param^1}(p_i)_{k-1}\rightarrow \tilde{\Delta}_{\param^1}(p_i)_k\rightarrow 0.$$
We then set $\tilde{\Delta}_{\param^1}(p_i):=\tilde{\Delta}_{\param^1}(p_i)_1$.
Note that this mirrors the construction of the projective objects in highest weight
categories.

\begin{Lem}\label{Lem:proj_tilde}
For $\lambda'$ equal to either $\lambda$ or to a Weil generic $\widehat{\lambda}\in \param^1$,
we have an epimorphism $P_{\lambda'}(p_i)\twoheadrightarrow \tilde{\Delta}_{\lambda'}(p_i)$.
It is an isomorphism when $\lambda'=\widehat{\lambda}$ and its kernel is filtered with
$\Delta_\lambda(p')$, where $p'>_\chi p$, when $\lambda'=\lambda$.
\end{Lem}
\begin{proof}
First of all, let us recall a standard fact. Let $A_{\param^1}$ be  an algebra
and $M_{\param^1},N_{\param^1}$ be finitely $A_{\param^1}$-modules that are free
in a neighborhood of a point $x\in \param^1$. Then $\Hom_{A_x}(M_x,N_x)=
\Hom_{A_{\param^1}}(M_{\param^1},N_{\param^1})_x$ and if we have
$\Ext^i_{A_{\param^1}}(M_{\param^1},N_{\param^1})_x\xrightarrow{\sim}
\Ext^i_{A_x}(M_x,N_x)$ for all $i=0,\ldots,k-1$, then
$\Ext^{k}(A_{\param^1})(M_{\param^1},N_{\param^1})_x\hookrightarrow
\Ext^k_{A_x}(M_x,N_x)$.

The algebra $\A_{\param^1}$ is countable dimensional.
It follows that the space $\Ext^2_{\A_{\param^1}}(M_{\param^1},N_{\param^1})$
is at most countable dimensional for any finitely generated $\A_{\param^1}$-modules
$M_{\param^1},N_{\param^1}$. In particular, the set of $\lambda^1\in \param^1$ such that
the maximal ideal of $\lambda^1$ has a nonzero annihilator in $\Ext^2_{\A_{\param^1}}(M_{\param^1},N_{\param^1})$
is countable. Therefore,  for a Weil generic
$\widehat{\lambda}\in \param^1$, we have $\Ext^1_{\A_{\param^1}}(M_{\param^1},N_{\param^1})_{\widehat{\lambda}}=
\Ext^1_{\A_{\widehat{\lambda}}}(M_{\widehat{\lambda}},N_{\widehat{\lambda}})$. From here
we deduce by induction that $\left(\tilde{\Delta}_{\param^1}(p_i)_k\right)_{\widehat{\lambda}}=
\tilde{\Delta}_{\widehat{\lambda}}(p_i)_k$ (where the object on the right hand side is
defined analogously to $\tilde{\Delta}_{\param^1}(p_i)_k$). This shows that $P_{\widehat{\lambda}}(p_i)=
\tilde{\Delta}_{\widehat{\lambda}}(p_i)$.

Let us consider the case of $\lambda'=\lambda$ now. We prove by induction
on $k$ that
\begin{equation}\label{eq:Ext_equality} \Ext^1_{\A_{\param^1}}(\tilde{\Delta}_{\param^1}(p_i)_k,
\Delta_{\param^1}(p_{k-1}))_\lambda=\Ext^1_{\A_{\lambda}}(\tilde{\Delta}_{\lambda}(p_i)_k,
\Delta_{\lambda}(p_{k-1})),\end{equation}
this will imply the claim of the proposition.
So far, we know that, first, the left hand side of (\ref{eq:Ext_equality}) is included into
the right hand side and, second, (\ref{eq:Ext_equality}) holds for $\lambda$
replaced with a Weil generic $\widehat{\lambda}\in \param^1$. To show (\ref{eq:Ext_equality}),
we need to show that the multiplicity
of $\Delta_{\lambda}(p_j)$ in $P_\lambda(p_i)$ (for $j<i$) coincides with
that of $\Delta_{\widehat{\lambda}}(p_j)$ in $P_{\widehat{\lambda}}(p_i)$. By the BGG
reciprocity, this is equivalent to $[\nabla_{\lambda}(p_j):L_\lambda(p_i)]=
[\nabla_{\widehat{\lambda}}(p_j): L_{\widehat{\lambda}}(p_i)]$. By the $\nabla$-analog of
Lemma \ref{Lem:Delta_filtr}, the right hand side coincides with the multiplicity
of $\bnabla_{\lambda}(p_i)$ in $\nabla_{\lambda}(p_j)$. By the properties (1)-(3)
of $\bnabla_\lambda(p_i)$, that coincides with the multiplicity
$[\nabla_{\lambda}(p_j):L_\lambda(p_i)]$.
\end{proof}

Similarly, we define the objects $\tilde{\nabla}_{\param^1}(p)$. A direct analog of
Lemma \ref{Lem:Delta_filtr} holds.

We proceed to proving that the pre-order $\leqslant_\chi$ defines a standardly
stratified structure on $\OCat_\nu(\A_\lambda)$. We start by proving two technical
lemmas.

\begin{Lem}\label{Lem:Ext_vanish}
We have $\dim \Ext^i_{\A_\lambda}(\tilde{\Delta}_\lambda(p), \bnabla_\lambda(p'))=\delta_{i0}\delta_{pp'}$.
\end{Lem}
\begin{proof}
Note that $\bnabla_{\lambda}(p')$ is quasi-isomorphic to a complex whose terms are
filtered with $\nabla_{\lambda}(p'')$ with $p''\sim_\chi p'$.
Using the $\nabla$-analog of Lemma \ref{Lem:Delta_filtr}, we see that
$\Ext^i_{\A_\lambda}(\widetilde{\Delta}_\lambda(p),\bnabla_\lambda(p'))=0$ for all $i$ unless
$p\sim_\chi p'$. The case of $p\sim_\chi p'$ follows from Lemma \ref{Lem:proj_tilde}.
\end{proof}

Similarly, we get
$\dim \Ext^i_{\A_\lambda}(\bDelta_\lambda(p), \tilde{\nabla}_\lambda(p'))=\delta_{i0}\delta_{pp'}$.

\begin{Lem}\label{Lem:stand_filtr}
Let $M\in \OCat_\nu(\A_\lambda)$ be such that $\Ext^i_{\A_\lambda}(M,\bnabla_\lambda(p))=0$
for all $p$. Then $M$ is filtered by $\tilde{\Delta}$'s.
\end{Lem}
\begin{proof}
Thanks to the $\nabla$-analog of Lemma \ref{Lem:Delta_filtr}, we see that $M$ is $\Delta$-filtered.
It is enough to prove the lemma in the case when $M$ is filtered by $\Delta_\lambda(p')$
with $p'\sim_\chi p$ (as $\{p'| p'\sim_\chi p\}$ is a poset in a highest weight order
on $\OCat_\nu(\A_\lambda)$). The  equality $\Ext^i_{\A_\lambda}(M,\bnabla_\lambda(p))=0$
is equivalent to $\Ext^i_{\OCat_p}(\pi_p(M),\pi_p(\bnabla_\lambda(p)))=0$,
where we write  $\mathcal{O}_p$ for  the subquotient highest weight category corresponding to
the interval $\{p'| p'\sim_\chi p\}$ and $\pi_p$ for the quotient
$\OCat_{\leqslant_\chi p}\twoheadrightarrow \OCat_p$.
Since the objects $\bnabla_\lambda(p')$
are simple in the subquotient, we deduce that the image of $M$ is projective. But, by Lemma \ref{Lem:prop_stand_stratif}, this precisely means that $M$ is the direct sum of the objects $\tilde{\Delta}_\lambda(p')$.
\end{proof}

Now we can complete the proof that $(\OCat_\nu(\A_\lambda),<_\chi)$ is a standardly
stratified category. By Lemma \ref{Lem:stand_filtr}, any projective in
$\OCat_\nu(\A_\lambda)$ is $\tilde{\Delta}$-filtered, which is (SS2). Similarly, every injective
is $\tilde{\nabla}$-filtered. It follows that the quotient functor $\pi_{\xi}$
maps the injective objects in the subcategory $\OCat_{\nu}(\leqslant_\chi p)\subset \OCat_{\nu}(\A_\lambda)$
to injective objects in the quotient $\OCat_\nu(\sim_\chi p)$. Equivalently,
the functor $\pi_\xi^!$ is exact, which is (SS1).

\subsubsection{Associated graded category}
To finish the proof of Proposition \ref{Prop:stand_stratif} we
need to check that $\operatorname{gr}\OCat_\nu(\A_\lambda)\cong
\OCat_{\nu}(\A_{\widehat{\lambda}})$. Note that the subquotient category
$\OCat_\lambda(\sim_\chi p)$ is equivalent to the category of right modules
over the algebra $B_\lambda:=\End_{\A_\lambda}(\bigoplus_{p'}\tilde{\Delta}_\lambda(p'))$,
where the summation is taken over all $p'\sim_\chi p$.
We can also consider the algebra $B_{\param^1}:=\End_{\A_{\param^1}}
(\bigoplus_{p'}\tilde{\Delta}_{\param^1}(p'))$, the algebra
$B_\lambda$ is the specialization of  $B_{\param^1}$ to $\lambda$ (and the same
is true for a Weil generic element $\widehat{\lambda}$). What remains to prove
is the following lemma.

\begin{Lem}\label{Lem:alg_iso}
We have an algebra isomorphism $B_\lambda\cong B_{\widehat{\lambda}}$ that respects
the primitive idempotents $e_{p'}$, where $p'\sim_\chi p$.
\end{Lem}
\begin{proof}
After passing to a principal Zariski open subset  $\param^0\subset \param^1$, we achieve that
the algebra $\B_{\param^0}$ is a free $\C[\param^0]$-module with a basis including the idempotents
$e_{p'}$ and compatible with the decomposition
$\bigoplus_{p',p''} e_{p'}\B_{\param^0}e_{p''}=\B_{\param^0}$. This gives rise to a
morphism $\param^0\rightarrow X$, where  $X$ denotes the variety of associative products
such that the elements $e_{p'}$ are idempotents. Isomorphisms correspond to a suitable
algebraic group action. What we need to prove is that a Zariski open subset of $\param^0$
maps to a single orbit. For this we note that we have a labeling preserving isomorphism
$\B_\lambda\cong \B_{\lambda+\chi'}$, where $\chi'$ is an integral element of $\Gamma$.
It follows that the elements $\lambda,\lambda+\chi'$ map to the same orbit. But the
set $\{\lambda+\chi'\}$ is Zariski dense in $\param^0$, which implies our claim.
\end{proof}

\subsection{Partial Ringel dualities}
In this section we discuss partial Ringel duality, i.e., Ringel duality for standardly stratified
categories\footnote{Preliminary versions of sections \ref{SS_tilt_sss}-\ref{SSS_ws_Ring_dual}
appeared in one of the first draft of our joint paper \cite{LW} with Ben Webster and were
later removed.}. The most important example of  Ringel duality functors comes from wall-crossing functors.

\subsubsection{Tilting and cotilting objects}\label{SS_tilt_sss}
An object in $\Cat\operatorname{-tilt}:=\Cat^{\Delta}\cap \Cat^{\bnabla}$ is called {\it tilting}.
We can construct an indecomposable tilting $T(\tau)$ similarly to the highest weight
case. Namely, let us order elements of $\mathcal{T}$, $\tau_1,\ldots,\tau_n$
in such a way that $\tau_i\geqslant \tau_j$ implies $i\leqslant j$.
Define the objects $T_j(\tau_i)$, where $j\leqslant i$, inductively as follows:
\begin{itemize}
\item $T_i(\tau_i):=\Delta(\tau_i)$.
\item Once $T_j(\tau_i)$ is constructed, for $T_{j-1}(\tau)$ we take the extension
of $\Ext^1(\Delta(\tau_{j-1}),T_j(\tau_i))\otimes \Delta(\tau_{j-1})$ by $T_j(\tau_i)$
corresponding to the identity endomorphism of  $\Ext^1(\Delta(\tau_{j-1}),T_j(\tau_i))$.
So we have an exact sequence
$$0\rightarrow T_j(\tau_i)\rightarrow T_{j-1}(\tau_i)\rightarrow \Ext^1(\Delta(\tau_{j-1}),T_j(\tau_i))\otimes \Delta(\tau_{j-1})\rightarrow 0.$$
\end{itemize}
We set $T(\tau_i):=T_1(\tau_i)$. From the construction of $T(\tau)$ and (\ref{Ext:vanish})
it is easy to see that $\Ext^1(\Delta(\tau'),T(\tau))=0$ for any $\tau'$. By (2)
of Lemma \ref{Lem:exts}, $T(\tau)\in \Cat^{\bnabla}$, so $T(\tau)$ is indeed tilting.

\begin{Lem}\label{Lem:indec_tilt}
The object $T(\tau)$ is indecomposable. Moreover, any indecomposable tilting object in
$\Cat$ is isomorphic to precisely one $T(\tau)$.
\end{Lem}
\begin{proof}
Note that by the construction of $T(\tau)$, the label $\tau$ is uniquely determined by the following property:
there is an embedding $\Delta(\tau)\hookrightarrow T(\tau)$ such that the cokernel is standardly filtered.
Moreover, Lemma \ref{Lem:exts}(2) implies that a direct summand of a standardly filtered object is standardly filtered.
These two observations imply both claims of the lemma.
\end{proof}

Applying this construction to $\Cat^{opp}$ we get {\it cotilting objects}.

Below we will need some further easy properties of tilting objects.

\begin{Lem}\label{Lem:tilting_top}
We have an epimorphism $T(\tau)\twoheadrightarrow \nabla(P_\xi(\tau))$, where $\xi=\varrho(\tau)$,
whose kernel lies in $\Cat^{\bnabla}_{<\xi}$.
\end{Lem}
\begin{proof}
In the proof we can assume that $\xi$ is the largest element of $\Xi$ (if not, we pass to the
standardly stratified subcategory $\Cat_{\leqslant \xi}$). Note that $\Ext^1(\bnabla(\tau'),\bnabla(\tau))=0$
if $\tau'<\tau$. It follows that we have a canonical exact sequence
$$0\rightarrow K\rightarrow T(\tau)\rightarrow C\rightarrow 0,$$
where $K\in \Cat_{<\xi}\cap \Cat^{\bnabla}$ and $C$ is filtered with successive
quotients of the form $\bnabla(\tau)$ with $\varrho(\tau)=\xi$. It remains to
prove that $C=\nabla(P_\xi(\tau))$. We have $\pi_\xi(C)=\pi_{\xi}(T(\tau))=\pi_{\xi}(\Delta(\tau))=P_{\xi}(\tau)$.
Then we can apply Lemma \ref{Lem:prop_stand_stratif} to $\Cat^{opp}$.
\end{proof}

\begin{Lem}\label{Lem:tilt_cover}
Let $M\in \Cat^{\Delta}$ and $N\in \Cat^{\bnabla}$. Then there is a tilting object $T_N$
with an epimorphism $T_N\twoheadrightarrow N$ whose kernel lies in $\Cat^{\bnabla}$.
Furthermore, any morphism $M\rightarrow N$ factors through $T_N$.
\end{Lem}
\begin{proof}
We can construct $T_N$ by taking the consecutive universal
extensions of $N$ by $\bnabla(\tau_i)\otimes \Ext^1(N,\bnabla(\tau_i))$ for $i$ ranging from
$1$ to $n=|\mathcal{T}|$ (we order labels in non-decreasing way with respect to $\leqslant$).
The claim about morphisms follows from Lemma
\ref{Lem:exts}(1) as the kernel of $T_N\twoheadrightarrow N$ is proper costandardly filtered.
\end{proof}

\subsubsection{Definition of Ringel duality}
Let $\Cat_1$ be a  standardly stratified category and $\Cat_2$ be a weakly standardly stratified
category.   By a Ringel duality data, we mean a pair $(\Ring,\theta)$ of
\begin{itemize}
\item A poset isomorphism $\theta:\Xi_1\xrightarrow{\sim}\Xi_2^{opp}$.
\item an  equivalence $\Ring:D^b(\Cat_1)\xrightarrow{\sim} D^b(\Cat_2)$
of triangulated categories that restricts to an equivalence $\Cat_1^{\bnabla}\xrightarrow{\sim} \Cat_2^{\bDelta}$
\item If $\Ring(\bnabla(\tau))=\bDelta(\tau')$, then $\varrho(\tau')=\theta\circ \varrho(\tau)$.
\end{itemize}
Note that an equivalence $\Cat_1^{\bnabla}\rightarrow \Cat_2^{\bDelta}$ of exact categories
automatically maps a proper costandard object to a proper standard one, and this induces
a bijection between the labelling sets of simples.

We say that $(\Cat_2,\Xi_2)$ is a Ringel dual of $(\Cat_1,\Xi_1)$. We call $\Ring$ the {\it Ringel duality
functor}.

\subsubsection{Existence}
Let $\Cat$ be a standardly stratified category. Let $T:=\bigoplus_{\tau\in \mathcal{T}}T(\tau)$
be the tilting generator. Consider the category $\Cat^\vee:=\End_{\Cat}(T)^{opp}\operatorname{-mod}$ and the functor
$\Ring:=\operatorname{RHom}_{\Cat}(T,\bullet):D^b(\Cat)\rightarrow D^b(\Cat^\vee)$.

The functor $\Ring$ is an equivalence.
Indeed, it is easy to see that the higher self-extensions
of $T$ vanish. So the target category for $\mathcal{R}$ is indeed $D^b(\Cat^\vee)$. This also shows
that $\mathcal{R}$ is a quotient functor. It is an equivalence because the objects $T(\tau)$
generate the triangulated category $D^b(\Cat)$. The  equivalence $\mathcal{R}$ is
exact on $\Cat^{\bnabla}$.

\begin{Prop}\label{Prop:Ringel_dual}
The category $\Cat^\vee$ is weakly standardly stratified with poset $\Xi^{opp}$ and the pair $(\Ring,\operatorname{id})$
is a Ringel duality data. Moreover, for any other Ringel dual category $\Cat'$ and Ringel duality data
$(\Ring',\theta')$, there is an equivalence $(\Phi,\phi):\Cat^\vee\rightarrow \Cat'$ of weakly
standardly stratified categories such that $\Ring'$ is isomorphic to $\Phi\circ \Ring$
and $\theta'=\phi$.
\end{Prop}

We will start by proving that $\Cat^\vee$ is indeed a weakly standardly stratified category
(a harder part) and then prove a uniqueness statement. After that we will briefly discuss conditions
under which $\Cat^\vee$ is standardly stratified and not just weakly standardly stratified.

\subsubsection{Weakly standardly stratified structure on $\Cat^\vee$}\label{SSS_ws_Ring_dual}
The following proposition shows that $\Cat^\vee$ is Ringel dual to $\Cat$.

\begin{Prop}
The category $\Cat^{\vee}$ has  a weakly standardly stratified structure
for the opposite preorder on $\mathcal{T}$. We have an identification
$\gr\Cat^\vee\cong \Cat^\vee$ and the functor $\Delta^\vee(\bullet):
\gr\Cat^\vee\rightarrow \Cat^\vee$ coincides with $\Hom_{\Cat}(T,\nabla(\bullet))$.
\end{Prop}
\begin{proof} The proof is in several steps.

{\it Step 1}.
  For an ideal $\Xi_0\subset \Xi$, the category $(\Cat_{\Xi_0})^\vee$ is
  naturally identified with a quotient of $\Cat^\vee$. The quotient
  functor is $M\mapsto Me_{\Xi_0}$, where $e_{\Xi_0}$ denote the central idempotent
  that is the projection from $T$ to $\bigoplus_{\tau\in \Lambda_0}T(\tau)$.
  So for a coideal $\Xi^0\subset \Xi$
  (=ideal $\Xi^0\subset \Xi^{opp}$) we can define the subcategory
  $(\Cat^{\vee})_{\Xi^0}\subset \Cat^\vee$ as the kernel of the
  projection $\Cat^\vee\twoheadrightarrow
  (\Cat_{\Xi\setminus\Xi^0})^\vee$.

{\it Step 2}.
Now we claim that $(\Cat^\vee)_\xi$ is naturally identified with
$\Cat_\xi$. By the definition of our filtration, we can
represent $(\Cat^{\vee})_\xi$ as the kernel of the quotient
$(\Cat_{\leqslant \xi})^\vee\twoheadrightarrow (\Cat_{<\xi})^\vee$.  So, in order to prove
the claim in the beginning of the paragraph, we can assume that $\xi$ is the largest element of $\Xi$.
We claim that the required equivalence is provided by the functor
$\iota:\operatorname{Hom}(T, \nabla_\xi(\bullet)):\Cat_\xi\rightarrow \Cat^\vee$.
The image of $\iota$ is contained in
$(\Cat^\vee)_\xi$ because $\Hom(T(\tau'),\nabla(\tau))=0$ for
$\tau'<\tau$. So $\iota$ is a functor $\Cat_\xi\rightarrow (\Cat^\vee)_\xi$.

Let us show that $\iota$ is an equivalence.
Consider the quotient $A_1$ of $\End(T)^{opp}$
by the 2-sided ideal of all morphisms that factor as $T\rightarrow T_{<\xi}
\rightarrow T$ for $T_{<\xi}\in \Cat_{<\xi}\operatorname{-tilt}$.
So $(\Cat^\vee)_\xi$ is just $A_1\operatorname{-mod}$.
Consider the object $R:=\bigoplus_{\tau\in \varrho^{-1}(\xi)}\nabla(P_\xi(\tau))$ that is a quotient
of $T$ in such a way that the kernel lies in $\Cat_{<\xi}$, see Lemma \ref{Lem:tilting_top}.
Clearly a morphism $T\rightarrow T$ induces a
morphism $R\rightarrow R$ and so we get a homomorphism $\operatorname{End}(T)^{opp}\rightarrow A_2$,
where $A_2:=\operatorname{End}(R)^{opp}$ so that $\Cat_\xi=A_2\operatorname{-mod}$. Note that
the homomorphism $\operatorname{End}(T)^{opp}\rightarrow A_2$ factors through $A_1$ because
we have $\operatorname{Hom}(\bigoplus_{\varrho(\tau)<\xi} T(\tau),R)=0$. Let $\varphi$
be the resulting homomorphism $A_1\rightarrow A_2$. It is straightforward from the construction
of $\iota:\Cat_\xi\rightarrow (\Cat^\vee)_\xi$ is just $\varphi^*$.

So we need to check that $\varphi$ is an isomorphism. By Lemma \ref{Lem:tilt_cover}, any homomorphism $T\rightarrow
\nabla(P_\xi(\tau))$ lifts to an endomorphism of $T$. So $\varphi$ is surjective. Now let
$\psi$ be an endomorphism of $T$ whose image lies in the kernel  $K$  of $T\twoheadrightarrow \nabla(P_\xi(\tau))$.
Recall that $K\in \Cat_{<\xi}\cap \Cat^{\bnabla}$. So there is a tilting object $T_K\in \Cat_{<\xi}$
with $T_K\twoheadrightarrow K$ and every morphism $T\rightarrow K$ factors through $T_K$.
It follows that any morphism $T\rightarrow K$ lies in the kernel of $\operatorname{End}(T)^{opp}\twoheadrightarrow A_1$.
Hence $\varphi$ is injective. This finishes the proof of the claim that $\iota$
is an equivalence and establishes an equivalence $\gr\Cat\cong \gr\Cat^\vee$ that we will be using
from now on.

{\it Step 3}.
Fix $\xi\in \Xi$ and set $\Xi^0:=\{\xi'\in \Xi| \xi'\geqslant\xi\},\Xi_0:=\Xi\setminus \Xi^0$.
Let us show that, under the identification $\Cat_\xi\cong \Cat^\vee_\xi$, the quotient
functor $(\pi^\vee)_\xi:(\Cat^\vee)_{\geqslant \xi}\twoheadrightarrow \Cat_\xi$ gets identified with
$\pi_{\Xi_0,\xi}(T\otimes_{\End(T)^{opp}}\bullet)$, where we write $\pi_{\Xi_0,\xi}$
for the quotient functor $\Cat_{\Xi_0}\twoheadrightarrow \Cat_\xi$.
Note that, for $N\in (\Cat^\vee)_{\geqslant \xi}$, the tensor product
$T\otimes_{\End(T)^{opp}}N$ lies in $\Cat_{\Xi_0}$ because $e_{\Xi_0}N=0$.
So the functor  $\pi_{\Xi_0,\xi}(T\otimes_{\End(T)^{opp}}\bullet)$ does make sense.
To show the coincidence of the functors we can replace $\Cat$ by $\Cat_{\leqslant \xi}$
(and so $\Cat^\vee$ will be replaced with a suitable quotient). Then $\pi_\xi(T\otimes_{\End(T)^{opp}}\bullet)$
is just a quasi-inverse of $\operatorname{Hom}(T,\nabla_\xi(\bullet))$ and our claim follows.


{\it Step 4}. An isomorphism $(\pi^\vee)_\xi(\bullet)\cong \pi_{\Xi_0,\xi}(T\otimes_{\End(T)^{opp}}\bullet)$  shows
 that the functor $\Delta^\vee_\xi:\Cat_\xi\rightarrow\Cat^\vee_{\geqslant \xi}$ defined by
  $\Delta_\xi^\vee(\bullet):=\Hom_{\Cat}(T,\nabla_{\xi}(\bullet))$ is
  left adjoint to the projection $\Cat^{\vee}_{\geqslant
    \xi}\twoheadrightarrow \Cat^\vee_{\xi}$. The functor $\Delta^\vee_\xi$ is exact.
    Moreover, the functor $\Hom(T,\bullet)$
  functor identifies $\Cat^{\bnabla}$ with
  $(\Cat^\vee)^{\bDelta}$. Under this identification, we have
  $$\bDelta^\vee(\tau)=\bnabla(\tau),\quad P^\vee(\tau)=T(\tau).$$
\end{proof}

\subsubsection{Uniqueness}
Now let $\Cat'$ be another Ringel dual category of $\Cat$, let $(\Ring',\theta')$
be the corresponding Ringel duality data. It follows that $\Ring\circ \Ring'^{-1}$
is an equivalence $D^b(\Cat')\xrightarrow{\sim} D^b(\Cat^\vee)$. A projective
$P'$ in $\Cat'$ satisfies $\Ext^1_{\Cat'}(P',\bDelta(\tau'))=0$ for all labels
$\tau'$. It follows that $\Ext^1_{\Cat}(\Ring'^{-1}(P'), \bnabla(\tau))=0$
for all labels $\tau$. So $\Ring'^{-1}(P')$ is $\Delta$-filtered. It is also
$\bnabla$-filtered by the definition of a Ringel duality functor. So $\Ring'^{-1}(P')$
is tilting and therefore $\Ring\circ \Ring'^{-1}$ is projective.
It follows that $\Ring\circ \Ring'^{-1}$ restricts to
$\Cat'\operatorname{-proj}\rightarrow \Cat\operatorname{-proj}$ and so
comes from an equivalence of abelian categories. This equivalence is automatically
an equivalence of weakly standardly stratified categories (and $\theta'$ is the corresponding
bijection of posets).

\subsubsection{When $\Cat^\vee$ is standardly stratified}
In general, it seems that $\Cat^\vee$ is only weakly standardly stratified. Clearly, the claim that $\Cat^\vee$
is standardly stratified is equivalent to the following claim:
\begin{itemize}
\item[(*)] All tilting objects in $\Cat$ admit a filtration with successive quotients of the form
$\nabla_\xi(P_\xi)$, where $P_\xi$ is a \underline{projective} object in $\Cat_\xi$ (instead of just some object that
is guaranteed by the condition of being tilting).
\end{itemize}
When all projectives in $\Cat_\xi$ are injective, (*) becomes
\begin{itemize}
\item[(**)] Tilting and co-tilting objects in $\Cat$ are the same.
\end{itemize}
For example, (**) is satisfied for tensor products of minimal categorifications
studied in \cite{LW}. In  order to see that one applies an analog of the inductive
construction of projective objects used in the proof of \cite[Theorem 6.1]{LW}
to construct tilting objects using the dual splitting procedure from \cite[Section 4.4]{LW}.
The inductive construction shows that all tilting objects are co-tilting.

\subsubsection{Wall-crossing functors}
Let $X,\nu,\lambda,\chi$ be as in \ref{SSS_main_result_stand_stratif}. Set $\lambda':=\lambda-N\chi$
for $N\gg 0$. We are going to describe the Ringel dual of $(\OCat_\nu(\A_\lambda),\leqslant_{\chi})$
and the corresponding Ringel duality functor.

\begin{Thm}\label{Thm:WC_part_Ringel}
The standardly stratified category $(\OCat_{\nu}(\A_{\lambda'}), \leqslant_{-\chi})$ is the
Ringel dual of $(\OCat_\nu(\A_\lambda),\leqslant_{\chi})$. The functor $\WC_{\lambda\leftarrow \lambda'}^{-1}$
is the corresponding Ringel duality functor.
\end{Thm}
\begin{proof}
We need to prove that there is a highest weight
equivalence $\iota:\gr \OCat_\nu(\A_\lambda)\xrightarrow{\sim} \OCat_{\nu}(\A_{\widehat{\lambda}})
\xleftarrow{\sim} \gr\OCat_{\nu}(\A_{\lambda'})$ such that we have a functorial isomorphism
$\WC_{\lambda'\leftarrow \lambda}\Delta_\chi(M)\cong \nabla_{-\chi}(\iota(M))$. Let $\lambda^-:=\lambda-N\theta$
for $N\gg 0$ so that $\lambda',\lambda^-$ lie in chambers that are opposite with respect to the face containing
$-\chi$. It follows that
\begin{equation}\label{eq:WC_decomp}
\WC_{\lambda\leftarrow \lambda^-}=\WC_{\lambda\leftarrow \lambda'}\circ \WC_{\lambda'\leftarrow \lambda^-}.\end{equation}

Since $\WC_{\lambda\leftarrow \lambda^-}^{-1}$ is the Ringel duality functor for the highest
weight structure on $\OCat_{\nu}(\A_\lambda)$, see Proposition \ref{Prop:cat_O}, we only need
to show that $\WC_{\lambda'\leftarrow \lambda^-}$ maps $\Delta_{\lambda^-}(p)$ to
$\Delta_{-\chi}\circ \nabla_{\widehat{\lambda}}(p)$. For the same reason as in
\cite[Proposition 3.2]{Cher_supp}, the object
$\WC_{\lambda'\leftarrow \lambda^-}\Delta_{\lambda^-}(p)$ has no higher homology and its class
in $K_0$ coincides with $[\Delta_{\lambda'}(p)]=[\Delta_{-\chi}\circ \nabla_{\widehat{\lambda}}(p)]$.
Because of this the functor $\WC_{\lambda'\leftarrow \lambda^-}$ restricts to an equivalence
$$D^b(\OCat_\nu(\A_{\lambda'})_{\leqslant_\chi p})\xrightarrow{\sim}
D^b(\OCat_\nu(\A_{\lambda^-})_{\leqslant_\chi p}).$$
Hence there is an equivalence $\underline{\WC}:
D^b(\OCat_{\nu}(\A_{\widehat{\lambda}^-}))\xrightarrow{\sim} D^b(\OCat_{\nu}(\A_{\widehat{\lambda}'}))$
such that $\pi_p\circ \WC^{-1}_{\lambda'\leftarrow \lambda^-}\cong \underline{\WC}\circ \pi_p$.
By adjunction, we get $\WC_{\lambda'\leftarrow \lambda^-}\circ \Delta_{-\chi}\cong
\Delta_{-\chi}\circ \underline{\WC}$. Note that $\underline{\WC}$ induces
the trivial map between the $K_0$-groups and sends $\Delta_{\widehat{\lambda}^-}(p)$
to an object. The highest weight orders for $\OCat_{\nu}(\A_{\widehat{\lambda}^-}),
\OCat_{\nu}(\A_{\widehat{\lambda}'})$ are opposite. It follows that $\underline{\WC}$
sends  $\Delta_{\widehat{\lambda}^-}(p)$ to $\nabla_{\widehat{\lambda}'}(p)$.
Composing $\underline{\WC}$ with $\WC_{\widehat{\lambda}'\leftarrow \widehat{\lambda}^-}^{-1}$
we get a required equivalence $\iota$.
\end{proof}

\begin{Rem}\label{Rem:RCA}
Proposition \ref{Prop:stand_stratif} and Theorem \ref{Thm:WC_part_Ringel} are still
true for categories $\mathcal{O}$ over Rational Cherednik algebras (since an analog
of (\ref{eq:WC_decomp}) is not known in that context, a priori, the Ringel duality functor
will be given by $\WC_{\lambda'\leftarrow \lambda^-}\circ \WC_{\lambda\leftarrow \lambda^-}^{-1}$).
We need to take the chamber structure defined by the $c$-order as in  \cite[Section 2.6]{rouq_der}
and consider wall-crossing functors from {\it loc.cit.} The proofs carry over to
the RCA situation more or less verbatim (in the proof of Lemma \ref{Lem:alg_iso}
we need to use equivalences from \cite[Proposition 4.2]{rouq_der} rather than
equivalences coming from localization theorem).
\end{Rem}

\begin{Rem}\label{Rem:wc_bij}
Note that the functor $\underline{\WC}$ in the proof of Theorem \ref{Thm:WC_part_Ringel}
coincides with the wall-crossing functor $\WC_{\widehat{\lambda}'\leftarrow \widehat{\lambda}^-}$
(where $\widehat{\bullet}$ means a Weil generic deformation of a parameter
along the face $\Gamma$) up to post-composing with an abelian equivalence
that is the identity on the level of $K_0$. The isomorphism
$\Delta\circ \underline{\WC}
\cong \WC_{\lambda'\leftarrow \lambda^-}\circ \Delta$ implies
that the wall-crossing bijections (see Section \ref{SS_wc_bij})
$\mathfrak{wc}_{\widehat{\lambda}'\leftarrow \widehat{\lambda}^-}$
and $\mathfrak{wc}_{\lambda'\leftarrow \lambda^-}$  coincide.
This result was established in \cite[Theorem 1.1(iii)]{cacti}
and \cite[Proposition 3.1]{Cher_supp} for special varieties $X$.
\end{Rem}

\begin{Rem}\label{Rem:Negut}
The isomorphism $\Delta\circ \underline{\WC}
\cong \WC_{\lambda'\leftarrow \lambda^-}\circ \Delta$ may be viewed as
a categorical analog of a result predicted by Maulik and Okounkov and proved by Negut
in the special case of affine type A quiver varieties.
This result, \cite[(3.21)]{Negut},  reduces the computation of a $K$-theoretic $R$-matrix $\mathbf{R}^{\pm}_{\mathbb{m}+r\theta, \mathbb{m}+(r+\varepsilon)\theta}$
to root quantum subalgebras.
\end{Rem}

\subsubsection{Example: cross-walling functors}
Let $X,\nu,\nu_0,\lambda$ be as in \ref{SSS_SSC_1param}. Assume, in addition,
that all components of $X^{\nu_0(\C^\times)}$ have the same dimension.
Let $d:=(\dim X-\dim X^{\nu_0(\C^\times)})/2$. Define a one-parameter subgroup
$\nu'$ as follows. Let $\nu_1$ be such that $N\nu_0+\nu_1$ lies in the same
open chamber as $\nu$. Then we set $\nu':=-N\nu_0+\nu_1$.

The following is a corollary  of Proposition \ref{Prop:CW_properties}
(compare with the proof of Theorem \ref{Thm:WC_part_Ringel}).
\begin{Prop}\label{Prop:CW_part_Ringel}
There is an abelian equivalence of the category $(\OCat_{\nu'}(\A_\lambda),\leqslant_{-\nu_0})$
with the Ringel dual of $(\OCat_\nu(\A_\lambda),\leqslant_{\nu_0})$ and is trivial on $K_0$
that intertwines the Ringel duality functor with $\CW_{\nu\leftarrow \nu'}^{-1}[-d]$.
\end{Prop}

\subsubsection{Prospective example: Webster's functors}
Let $\mathcal{V}_1,\ldots,\mathcal{V}_k$ be as in \ref{SSS_cat_tens_prod}. Fix a composition
$k=k_1+\ldots+k_r$. Let $w_0$ be the longest element of $S_k$ and $w_0'$
be the longest element of $S_{k_1}\times S_{k_2}\times\ldots\times
S_{k_r}$. Set $\sigma=w_0'w_0^{-1}$. Let $\mathcal{V}_1\otimes\ldots \otimes \mathcal{V}_k$ denote
the categorical tensor product of the categories $\mathcal{V}_1,\ldots,\mathcal{V}_k$.

For $s\in S_k$, Webster in \cite{Webster_tensor} defined right $t$-exact equivalences
$\mathcal{T}_{s}: D^b(\mathcal{V}_1\otimes\ldots\otimes \mathcal{V}_k)\xrightarrow{\sim}
D^b(\mathcal{V}_{s(1)}\otimes\ldots\otimes \mathcal{V}_{s(k)})$. He showed that these
equivalences  give rise to a braid group action. The category
$\mathcal{V}_{\sigma(1)}\otimes \ldots\otimes \mathcal{V}_{\sigma(k)}$
should be the  Ringel dual of $\mathcal{V}_1\otimes\ldots\otimes \mathcal{V}_k$
with respect to a suitable standardly stratified structure.
The functor $\mathcal{T}_{\sigma}^{-1}$ should be the Ringel duality functor.



\begin{thebibliography}{99}
\bibitem[ADL]{ADL} I. Agoston, V. Dlab, E. Lukasz. {\it Stratified algebras}.
C. R. Math. Acad. Sci. Soc. R. Can.  20  (1998),  no. 1, 22–28.
\bibitem[BaGi]{BarGin} V. Baranovsky, V. Ginzburg. In preparation.
\bibitem[BeKa]{BK}  R. Bezrukavnikov, D. Kaledin. {\it Fedosov quantization in the algebraic context}.
Moscow Math. J. 4 (2004), 559-592.
\bibitem[BL]{BL} R. Bezrukavnikov, I. Losev, {\it Etingof conjecture for quantized quiver varieties}. arXiv:1309.1716.
\bibitem[BoKr]{BoKr} W. Borho, H. Kraft. {\it \"{U}ber die Gelfand-Kirillov-Dimension}. Math. Ann. 220(1976), 1-24.
\bibitem[BLPW]{BLPW} T. Braden, A. Licata, N. Proudfoot, B. Webster, {\it Quantizations of conical symplectic resolutions II: category O and symplectic duality}. arXiv:1407.0964. To appear in Ast\'{e}risque.
\bibitem[BPW]{BPW} T. Braden, N. Proudfoot, B. Webster, {\it Quantizations of conical symplectic resolutions I:
local and global structure}. arXiv:1208.3863. To appear in Ast\'{e}risque.
\bibitem[CPS]{CPS}  E. Cline, B. Parshall, and L. Scott, {\it Stratifying endomorphism algebras}, Mem. Amer. Math. Soc. 124 (1996), no. 591, viii+119.
\bibitem[CB]{CB} W.~Crawley-Boevey,  {\it Geometry of the moment map for representations of quivers,} Comp. Math. {\bf 126} (2001), 257--293.
\bibitem[K1]{Kaledin} D. Kaledin. {\it Symplectic singularities from the Poisson point of
view}. J. Reine Angew. Math. 600(2006), 135-156.
\bibitem[K2]{Kaledin_survey}
D. Kaledin, {\it Geometry and topology of symplectic resolutions}. Algebraic geometry, Seattle 2005. Part 2, 595-628, Proc. Sympos. Pure Math., 80, Part 2, Amer. Math. Soc., Providence, RI, 2009.
\bibitem[L1]{HC} I. Losev. {\it Finite dimensional representations of
W-algebras}. Duke Math J. 159(2011), n.1, 99-143.
\bibitem[L2]{W_prim} I. Losev. {\it Primitive ideals in W-algebras of type A}. J. Algebra, 359 (2012), 80-88.
\bibitem[L3]{quant} I. Losev. {\it Isomorphisms of quantizations via quantization of resolutions}. Adv. Math. 231(2012), 1216-1270.
\bibitem[L4]{count_affine} I. Losev. {\it Etingof conjecture for quantized quiver varieties II: affine quivers}.
arXiv:1405.4998.
\bibitem[L5]{rouq_der} I. Losev. {\it Derived equivalences for Rational Cherednik algebras}. arXiv:1406.7502.
\bibitem[L6]{B_ineq} I. Losev. {\it Bernstein inequality and holonomic modules} (with a joint appendix
by I. Losev and P. Etingof). arXiv:1501.01260.
\bibitem[L7]{CWR} I. Losev. {\it On categories $\mathcal{O}$ for quantized symplectic resolutions}. arXiv:1502.00595.
\bibitem[L8]{cacti} I. Losev. {\it Cacti and cells}. arXiv:1506.04400.
\bibitem[L9]{Cher_supp} I. Losev. {\it Supports of simple modules in cyclotomic Cherednik categories O}.
arXiv:1509.00526.
\bibitem[LW]{LW} I. Losev, B. Webster, {\it On uniqueness of tensor products of irreducible categorifications}.
arXiv:1303.1336.  Selecta Math. 21(2015), N2, 345-377.
\bibitem[Nak]{Nakajima} H. Nakajima. {\it Instantons on ALE spaces, quiver varieties and Kac-Moody algebras}.
Duke Math. J. 76(1994), 365-416.
\bibitem[Nam]{Namikawa13} Y. Namikawa, {\it Poisson deformations and Mori dream spaces},
arXiv:1305.1698.
\bibitem[Ne]{Negut} A. Negut. {\it Quantum Algebras and Cyclic Quiver Varieties}.
arXiv:1504.06525.
\bibitem[R1]{Rouquier_ICM} R. Rouquier, {\it Derived equivalences and finite dimensional algebras}. Proceedings of
ICM 2006.
\bibitem[R2]{Rouquier_2Kac} R. Rouquier, {\it 2-Kac-Moody algebras}. arXiv:0812.5023.
\bibitem[S]{Sumihiro} H. Sumihiro, {\it Equivariant completion}.
J. Math. Kyoto Univ.  14  (1974), 1–28.
\bibitem[W]{Webster_tensor} B. Webster, {\it Knot invariants and higher representation theory}.
arXiv:1309.3796. To appear in Mem. Amer. Math. Soc.
\end{thebibliography}
\end{document}